\documentclass[10.5pt]{amsart}

\usepackage{amsmath, amssymb, amsfonts}
\usepackage{graphicx, overpic}
\usepackage{color}
\usepackage{enumerate}
\usepackage{multicol}

\usepackage[margin=1.2in]{geometry}

\numberwithin{equation}{section}
\numberwithin{figure}{section}

\theoremstyle{plain}
\newtheorem{thm}{Theorem}[section]
\newtheorem{lemma}[thm]{Lemma}
\newtheorem{prop}[thm]{Proposition}
\newtheorem{cor}[thm]{Corollary}
\theoremstyle{definition}
\newtheorem{defn}[thm]{Definition}
\newtheorem{example}[thm]{Example}
\newtheorem{remark}[thm]{Remark}

\newtheorem{notation}[thm]{Notation}

\newtheorem{construction}[thm]{Construction}
\newtheorem{outline}[thm]{Outline}

\newcommand{\Z}{\mathbb{Z}}
\newcommand{\R}{\mathbb{R}}
\newcommand{\N}{\mathbb{N}}

\newcommand{\Hy}{\mathbb{H}}

\newcommand{\cB}{\mathcal{B}}
\newcommand{\cC}{\mathcal{C}}

\newcommand{\cG}{\mathcal{G}}
\newcommand{\cH}{\mathcal{H}}

\newcommand{\cM}{\mathcal{M}}
\newcommand{\cN}{\mathcal{N}}

\newcommand{\cQ}{\mathcal{Q}}

\newcommand{\cS}{\mathcal{S}}

\newcommand{\cV}{\mathcal{V}}

\newcommand{\cZ}{\mathcal{Z}}

\newcommand{\gL}{\Lambda}
\newcommand{\G}{\Gamma}

\newcommand{\Int}{\operatorname{Int}}

\newcommand{\la}{\langle}
\newcommand{\ra}{\rangle}
\newcommand{\p}{\partial}

\definecolor{amethyst}{rgb}{0.6, 0.4, 0.8}

\newcommand{\hide}[1]{}

\title{Detecting a subclass of torsion-generated groups}
\author{Emily Stark}
\address{Department of Mathematics\\
Technion - Israel Institute of Technology \\
Haifa 32000 \\
Israel }
\email{emily.stark@technion.ac.il}
\date{\today}
\thanks{The author is thankful for support from the Azrieli Foundation. The author was partially supported at the Technion by a Zuckermann STEM Leadership Postdoctoral Fellowship.}

\begin{document}

\begin{abstract}
   We classify the groups quasi-isometric to a group generated by finite-order elements within the class of one-ended hyperbolic groups which are not Fuchsian and whose JSJ decomposition over two-ended subgroups does not contain rigid vertex groups. To do this, we characterize which JSJ trees of a group in this class admit a cocompact group action with quotient a tree. The conditions are stated in terms of two graphs we associate to the degree refinement of a group in this class. We prove there is a group in this class which is quasi-isometric to a Coxeter group but is not abstractly commensurable to a group generated by finite-order elements. Consequently, the subclass of groups in this class generated by finite-order elements is not quasi-isometrically rigid.  We provide necessary conditions for two groups in this class to be abstractly commensurable. We use these conditions to prove there are infinitely many abstract commensurability classes within each quasi-isometry class of this class that contains a group generated by finite-order elements. 
\end{abstract}

\maketitle

\section{Introduction} \label{sec:intro}

  The large-scale geometry type of a finitely generated group does not depend on whether the group contains elements of finite order; every finitely generated group is quasi-isometric to a group that contains torsion. However, quasi-isometry classes that contain groups generated by finite-order elements are distinguishable. Torsion-generated groups, such as Coxeter groups, play an important role in geometric group theory. Background is given by Davis~\cite{davis}. An interesting problem is to determine which finitely generated groups are quasi-isometric to a group generated by finite-order elements. In this paper, we solve this problem within a certain class of hyperbolic groups. 
  
  A natural approach to this problem begins by decomposing the group using a graph of groups decomposition. Dunwoody~\cite{dunwoody} proved every finitely presented group admits a maximal splitting as a graph of groups with finite edge groups, and Papasoglu--Whyte~\cite{papasogluwhyte} proved for a infinite-ended finitely presented group, the set of quasi-isometry classes of the one-ended vertex groups in this graph of groups decomposition is a complete quasi-isometry invariant. Therefore, an infinite-ended finitely presented group is quasi-isometric to a group generated by finite-order elements if and only if each of the one-ended vertex groups in this graph of groups decomposition is quasi-isometric to a group generated by finite-order elements. Thus, for finitely presented groups, the problem reduces to the case the group is one-ended. 
      
  Rips--Sela~\cite{ripssela} proved if $G$ is a one-ended finitely presented group that is not Fuchsian, then there is a canonical graph of groups decomposition of~$G$, called the {\it JSJ decomposition of $G$}, with edge groups that are $2$-ended and vertex groups of three types: $2$-ended; {\it maximally hanging Fuchsian}; and, quasi-convex {\it rigid} vertex groups not of the first two types. In this paper, we follow the language and structure of the JSJ decomposition due to Bowditch~\cite{bowditch}  for one-ended hyperbolic groups that are not Fuchsian. We characterize the groups quasi-isometric to a group generated by finite-order elements within the class $\cC$ of $1$-ended hyperbolic groups that are not Fuchsian and whose JSJ decomposition does not contain rigid vertex groups. 
      
  The isomorphism type of the Bass--Serre tree of the JSJ decomposition of a group in $\cC$ is a complete quasi-isometry invariant, as shown by Malone \cite{malone} for a subclass of groups in $\cC$ called geometric amalgams of free groups and by Cashen--Martin \cite{cashenmartin} in the general setting; see also related work of Dani--Thomas \cite{danithomas}. Furthermore, there is a one-to-one correspondence between isomorphism types of JSJ trees of groups in $\cC$ and (equivalence classes of) certain finite matrices called {\it degree refinements}, which are algorithmically computed from the JSJ~decomposition.   
      
  A group generated by finite-order elements does not surject onto $\Z$. Consequently, if a group $G \in \cC$ is generated by finite-order elements, then the underlying graph of the JSJ decomposition of $G$ is a tree; we call this graph the {\it JSJ graph of $G$}. Conversely, Dani--Stark--Thomas \cite[Theorem~1.16]{danistarkthomas} proved if a quasi-isometry class in $\cC$ contains a group whose JSJ graph is a tree, then the quasi-isometry class contains a right-angled Coxeter group. Therefore, classifying the quasi-isometry classes within $\cC$ which contain a group generated by finite-order elements is equivalent to classifying the JSJ trees of a group in $\cC$ which admit a cocompact group action with quotient a tree. To accomplish this, we introduce two graphs associated to the degree refinement of a group $G \in \cC$: the {\it graph of blocks} of $G$ and the {\it augmented graph of blocks} of $G$. The graph of blocks of $G$ has vertex set in one-to-one correspondence with orbits of vertices in the JSJ tree of $G$ under the action of the full isometry group of the tree. The augmented graph of blocks of $G$ is the graph with fewest vertices and degree refinement equivalent to the degree refinement of $G$. See Section~\ref{sec:deg_ref} for more details. The first main result of the paper is the following.

  \begin{thm} \label{thm_sec1_qi_char}
   Let $G \in \cC$. The following are equivalent.
   \begin{enumerate}
    \item The group $G$ is quasi-isometric to a right-angled Coxeter group.
    \item The group $G$ is quasi-isometric to a group generated by finite-order elements.
    \item The group $G$ is quasi-isometric to a group with JSJ graph a tree.
    \item The degree refinement of $G$ satisfies the two conditions:
      \begin{itemize}
	\item[(M1)] The graph of blocks of $G$ is a tree.
	\item[(M2)] The augmented graph of blocks of $G$ has no $2$-cycles at even distance bounded by Type I vertices. 
      \end{itemize}
   \end{enumerate}
  \end{thm}

  The equivalence of Conditions (1)-(3) was established in \cite{danistarkthomas}; see Section~\ref{sec:char}. We prove here that (1)-(3) are equivalent to (4), a pair of conditions defined in Section~\ref{sec:deg_ref} that can be easily verified. Thus, from the underlying graph of the JSJ decomposition of a group in $\cC$, one may determine whether Conditions (1)-(3) hold. The graphs in Condition (4) are quasi-isometry invariants (Corollary~\ref{cor_same_block_gr}), and we prove Conditions (M1) and (M2) hold for a group in $\cC$ with JSJ graph a tree. Conversely, if Conditions (M1) and (M2) hold, we perform a finite series of moves on the augmented graph of blocks of the group to obtain a finite tree with an equivalent degree refinement. 
  
  Two groups are {\it abstractly commensurable} if the groups contain finite-index subgroups that are isomorphic; two finitely generated groups that are abstractly commensurable are quasi-isometric. Hence, the result in Theorem~\ref{thm_sec1_qi_char} gives a partial answer to \cite[Question 1.3]{danistarkthomas}, which asks which geometric amalgams of free groups are abstractly commensurable to a right-angled Coxeter group. However, the next result proves the set of groups quasi-isometric to right-angled Coxeter groups in~$\cC$ is strictly larger than the set of groups abstractly commensurable to a right-angled Coxeter group in~$\cC$.

  \clearpage
  
   \begin{thm} \label{thm:sec_1_qi_notAC}
    There exists $G \in \cC$ for which the following holds.
      \begin{enumerate}
       \item The group $G$ is quasi-isometric to a right-angled Coxeter group. 
       \item The group $G$ is not abstractly commensurable to any group with JSJ graph a tree. In particular, $G$ is not abstractly commensurable to any group generated by finite-order elements.
      \end{enumerate}
   \end{thm}
   
   To prove Theorem~\ref{thm:sec_1_qi_notAC}, we construct $G$ in Construction~\ref{const:not_AC} as the fundamental group of a union of surfaces with boundary glued together along their boundary components.     While the JSJ tree of $G \cong \pi_1(X)$ has a finite quotient which is a tree, there is an asymmetry in the Euler characteristics of certain subsurfaces in the space $X$ that yields the commensurability result. 
    
   A class of groups $\cG$ is {\it quasi-isometrically rigid} if every group quasi-isometric to a group in $\cG$ is abstractly commensurable to a group in $\cG$. (A slightly different notion of quasi-isometric rigidity requires every group quasi-isometric to a group in $\cG$ to be {\it virtually isomorphic} to a group in $\cG$. These notions are equivalent within $\cC$ as groups in $\cC$ are virtually torsion-free; see \cite[Observation 3.1]{haissinskypaoluzziwalsh} and the book by Dru\c{t}u--Kapovich \cite{drutukapovich} for background.)
    The class of Coxeter groups is not quasi-isometrically rigid. For example, Burger--Mozes \cite{burgermozes} provided examples of (non-hyperbolic) infinite simple groups which act geometrically on the product of two finite-valence trees. Such groups have no finite-index subgroups, yet these groups are quasi-isometric to the direct product of two free groups of rank greater than one, and, hence, are quasi-isometric to a right-angled Coxeter group with defining graph a complete bipartite graph with vertex sets of size greater than two. The class of groups $\cC$ is quasi-isometrically rigid by the construction of the JSJ decomposition given by Bowditch~\cite{bowditch}; see \cite[Observation 3.1]{haissinskypaoluzziwalsh} for a related result. Theorem~\ref{thm:sec_1_qi_notAC} proves the set of right-angled Coxeter groups within the class $\cC$ is not quasi-isometrically rigid, and we also have the following corollary.   
    
  \begin{cor}
   The subclass of groups in $\cC$ which have JSJ graph a tree is not quasi-isometrically rigid. The subclass of groups in $\cC$ which are generated by finite-order elements is not quasi-isometrically rigid. 
  \end{cor}  
  
  There are classes of right-angled Coxeter groups which are quasi-isometrically rigid; simple examples include virtually-free right-angled Coxeter groups and right-angled Coxeter groups which act properly and cocompactly by isometries on the hyperbolic plane \cite{tukia,gabai,cassonjungreis}. Determining the classes of right-angled Coxeter groups which are quasi-isometrically rigid is an interesting problem. A natural focus is classes for which quasi-isometry invariants or classification is known; see \cite{danithomas-div,caprace-erra,charneysultan,behrstockhagensisto,behrstockhagensisto-quasiflats,levcovitz16,levcovitz17,haulmarknguyentran}.  
  
  The abstract commensurability classification within $\cC$ remains open. This classification may be an important step in resolving \cite[Question 1.3]{danistarkthomas}. Partial results are given by Crisp--Paoluzzi \cite{crisppaoluzzi}, Malone~\cite{malone}, the author \cite{stark}, and Dani--Stark--Thomas \cite{danistarkthomas}. These known results impose strong conditions on the subclass of groups considered: for example, they require the diameter of the JSJ graph to be at most $4$. In Proposition~\ref{prop_block_comm} and Proposition~\ref{prop_matching_comm} we prove two necessary conditions for commensurability for any geometric amalgam of free groups in $\cC$ with JSJ graph a tree. The commensurability invariants are the commensurability classes of two vectors whose entries record the sum of the Euler characteristics of certain vertex groups.   
  As a consequence, we prove the following theorem.
 
  \begin{thm} \label{thm:sec1_qi_vs_ac}
   There are infinitely many abstract commensurability classes within every quasi-isometry class in $\cC$ that contains a group generated by finite-order elements. 
  \end{thm}

 The conclusion of Theorem~\ref{thm:sec1_qi_vs_ac} is not surprising based on the previous results on commensurability for groups in $\cC$. The novelty of the result comes from removing the hypothesis that the JSJ graphs have small diameter. Moreover, previous results strongly use a topological rigidity theorem of Lafont~\cite{lafont}, which applies only to geometric amalgams of free groups in $\cC$. The result in Proposition~\ref{prop_block_comm} does not require this hypothesis.
  
 The main questions addressed in this paper, whether a given group is quasi-isometric or abstractly commensurable to a right-angled Coxeter group (or, more generally, a group generated by finite-order elements), may be viewed as coarse versions of the problem of determining whether a given group is a right-angled Coxeter group (or a group generated by finite-order elements); see, for example, \cite{charneyruanestambaughvijayan}, \cite{cunninghameisenbergpiggottruane}. 
  
 \subsection*{Acknowledgments}  The author is thankful for helpful discussions with Pallavi Dani, Misha Kapovich, Michah Sageev, Anne Thomas, and Genevieve Walsh. The author is grateful to the anonymous referee for thoughtful comments and corrections.  
 
 \subsection*{Outline} Preliminaries are given in Section~\ref{sec:prelims}. The degree refinement of a group in $\cC$ and related graphs are defined in Section~\ref{sec:deg_ref}. Theorem~\ref{thm_sec1_qi_char} is proven in Section~\ref{sec:char_qi}. Section~\ref{sec_comm_classes} contains the proof of Theorem~\ref{thm:sec_1_qi_notAC}. Section~\ref{sec:comm_cond} contains the proof of necessary conditions for commensurability. Theorem~\ref{thm:sec1_qi_vs_ac} is proven in Section~\ref{sec:qi_vs_ac}. 
 
\section{Preliminaries} \label{sec:prelims}

\subsection{Graph theory}

  In this section, we record relevant graph-theoretic terminology and establish notation. Most graphs we consider are unoriented, and we view these graphs as CW-complexes.

  Let $\Lambda = (V(\Lambda), E(\Lambda))$ be a graph, where $V(\gL)$ is the vertex set of $\gL$ and $E(\gL)$ is the edge set of~$\gL$. If $e = (u,v) \in E(\gL)$, we say $e$ is {\it incident} to the vertices $u$ and $v$, and we say $u$ and $v$ are {\it adjacent} vertices. A graph is {\it bipartite} if $V(\Lambda)$ is the disjoint union of two nonempty subsets $V(\gL) = V_1 \sqcup V_2$ such that every edge of $\gL$ is incident to exactly one element of $V_1$ and exactly one element of $V_2$. A {\it tree} is a connected graph that does not contain an embedded cycle. A {\it leaf} of a graph is a vertex of valence one. An {\it oriented graph} $\gL$ consists of a vertex set $V(\gL)$, an edge set $E(\gL)$, and maps $i:E(\gL) \rightarrow V(\gL)$ and $t:E(\gL) \rightarrow V(\gL)$. For each edge $e \in E(\Lambda)$, we refer to $i(e)$ as the {\it initial vertex} of $e$ and $t(e)$ as the {\it terminal vertex} of $e$. If $S \subset V(\gL)$, the subgraph of $\gL$ {\it induced by} the vertices in $S$ is the subgraph whose vertex set is $S$ and whose edge set consists of all edges in $E(\gL)$ that have both endpoints in $S$.

\subsection{JSJ decomposition and the class of groups considered} \label{sec:jsj}

   \begin{defn} \label{def:carry}
   A \emph{graph of groups} $\mathcal{G}$ is a graph $\gL = (V(\gL), E(\gL))$ with a \emph{vertex group} $G_v$ for each $v \in V(\gL)$, an \emph{edge group} $G_e$ for each $e \in E(\gL)$, and \emph{edge maps}, which are injective homomorphisms $\Theta^{\pm}_e: G_e \rightarrow G_{\pm e}$ for each $e =(-e,+e) \in E(\gL)$. The graph $\gL$ is called the {\it underlying graph} of $\cG$. 
   
   A \emph{graph of spaces} associated to a graph of groups $\mathcal{G}$ is a space $X$ with a graph $\gL$ constructed from a pointed \emph{vertex space} $(X_v, x_v)$ for each $v \in V(\gL)$ with $\pi_1(X_v,x_v) = G_v$, a pointed {\it edge space} $(X_e, x_e)$ for each $e \in E(\gL)$ such that $\pi_1(X_e,x_e) = G_e$, and maps $\theta^{\pm}_e: (X_e,x_e) \rightarrow (X_{\pm e}, x_{\pm{e}})$ such that $(\theta^{\pm}_e)_* = \Theta^{\pm}_e$.
   The space $X$ is the union $$\left(\bigsqcup_{v \in V(\gL)} X_v \bigsqcup_{e \in E(\gL)} \left( X_e \times [-1,1]\right) \right) \; \Big/ \; \left\{(x,\pm 1) \sim \theta_e^{\pm}(x) \mid (x, \pm 1) \in X_e\times [-1,1] \right\}. $$
   The \emph{fundamental group} of the graph of groups $\mathcal{G}$ is $\pi_1 (X)$.
   A group $G$ \emph{splits as graph of groups} if $G$ is the fundamental group of a non-trivial graph of groups.
  \end{defn}

  \begin{defn} \label{def:fuch} 
   A {\it Fuchsian group} is a non-elementary finitely generated group which acts properly discontinuously on the hyperbolic plane $\Hy^2$. 
   \end{defn}

  \begin{remark} \label{rem:fuch_orb}
   The action of a Fuchsian group $G$ on the hyperbolic plane need not be faithful, but, the kernel of the action is finite. This kernel is the unique maximal finite normal subgroup of $G$. Thus, the quotient $X = \Hy^2/G$ is a canonically defined ({\it ineffective}) orbifold. In particular, $X$ admits a decomposition as a finite cell complex so that each cell $\sigma$ of $X$ is equipped with a finite isotropy group $K_{\sigma}$ which is isomorphic to the stabilizer of each lift of $\sigma$ in $\Hy^2$. 
  \end{remark}
  
  \begin{defn} 
   Let $G$ be a Fuchsian group so that $G \cong \pi_1^{orb}(X)$, where $X = \Hy^2 /G$ is an (ineffective) orbifold. Realize $X$ with a cell decomposition so that each cell $\sigma$ of $X$ has a well-defined isotropy subgroup $K_{\sigma}$. The {\it Euler characteristic of $G$}, denoted $\chi(G)$, is
   \[ \chi(G) = \chi(X) = \sum_{\sigma \text{ cell in X}} (-1)^{\text{dim} (\sigma)} \frac{1}{|K_{\sigma}|}. \]
  \end{defn}
  
  \begin{lemma} \label{lemma:chi_mult}
   Let $G$ and $H$ be Fuchsian groups. If $H$ is an index-$d$ subgroup of $G$, then $d\cdot\chi(G) = \chi(H)$.
  \end{lemma}
    
  For background on orbifolds, see Kapovich~\cite{kapovich} and Ratcliffe~\cite{ratcliffe}. The proof of Lemma \ref{lemma:chi_mult} is also given by the more general theory presented by Brown \cite[Chapter IX-7]{brown}.
  
  \begin{defn}
   A {\it bounded Fuchsian group} is a Fuchsian group that is convex cocompact but not cocompact. The convex core of the quotient is a compact orbifold with non-empty boundary consisting of a disjoint union of compact $1$-orbifolds. The {\it peripheral subgroups} are the maximal two-ended subgroups which project to the fundamental groups of the boundary $1$-orbifolds. 
   A {\it hanging Fuchsian} subgroup $H$ is a virtually-free quasiconvex subgroup together with a collection of {\it peripheral} two-ended subgroups, which arise from an isomorphism of $H$ with a bounded Fuchsian group. A {\it full quasiconvex subgroup} of a group $G$ is a subgroup that is not a finite-index subgroup of any strictly larger subgroup of $G$.
  \end{defn}
  
  \begin{thm}\cite[Thm 0.1]{bowditch}
    Let $G$ be a one-ended hyperbolic group that is not Fuchsian.
    There is a canonical \emph{JSJ decomposition} of $G$ as the fundamental group of a graph of groups such that each edge group is 2-ended and each vertex group is either (1)  $2$-ended; (2) maximal hanging Fuchsian; or, (3) a maximal quasi-convex subgroup not of type (2). These types are mutually exclusive, and no two vertices of the same type are adjacent. Every vertex group is a full quasi-convex subgroup. Moreover, the edge groups that connect to any given vertex group of type (2) are precisely the peripheral subgroups of that group.      
  \end{thm}
    
  \begin{defn}
   Let $G$ be a one-ended hyperbolic group that is not Fuchsian.
   The {\it JSJ tree} of $G$ is the Bass--Serre tree of the JSJ decomposition of~$G$. The {\it JSJ graph} of $G$ is the underlying graph of the JSJ decomposition of $G$. 
  \end{defn}
    
  \begin{defn} (Class of groups considered.)
   Let $\cC$ denote the class of one-ended hyperbolic groups which are not Fuchsian and for which the JSJ decomposition has no vertex groups of type (3). 
  \end{defn}
  
  \begin{remark}
   If $G \in \cC$, then the JSJ graph of $G$ is bipartite. 
  \end{remark}

\subsection{Hyperbolic $P$-manifolds}

  We make use of the following subclass of groups in $\cC$. An example of a space defined below is given in Figure~\ref{figure:not_AC_tree}.

  \begin{defn} \label{def:geom_amal}
   A {\it $2$-dimensional hyperbolic $P$-manifold} $X$ is a space with a graph of spaces decomposition over a finite oriented graph $\gL$ with the following properties.
   \begin{enumerate}
    \item The underlying graph $\gL$ is bipartite with vertex set $V(\Lambda) = V_1 \sqcup V_2$ and edge set $E(\gL)$ such that each edge $e \in E(\gL)$ has $i(e) \in V_1$ and $t(e) \in V_2$. 
    \item For each $u \in V_1$, the vertex space $X_u$ is a copy of the circle $S^1$, and $x_u$ is a point on $X_u$. For each $v \in V_2$, the vertex space $X_v$ is a connected surface with negative Euler characteristic and non-empty boundary, and $x_v$ is a point on $X_v$.
    \item For each edge $e \in E(\gL)$, the edge space $X_e$ is a copy of the circle $S^1$, and $x_e$ is a point on $X_e$. If $e = (-e,+e)$ with $-e = i(e)$ and $+e = t(e)$, the map $\theta_e^-:(X_e,x_e) \rightarrow (X_{i(e)}, x_{i(e)})$ is a homeomorphism, and the map $\theta_e^+:(X_e,x_e) \rightarrow (X_{t(e)}, x_{t(e)})$ is a homeomorphism onto a boundary component of $X_{t(e)}$. 
     \item Each vertex $u \in V_1$ has valance at least three. Given any vertex $v \in V_2$, for each boundary component $B$ of $X_v$, there exists an edge $e$ with $t(e) = v$, such that the associated edge map identifies $X_e$ with $B$. The valance of $v$ is the number of boundary components of $X_v$. 
    \end{enumerate}
  The fundamental group of a $2$-dimensional hyperbolic $P$-manifold is a \emph{geometric amalgam of free groups} and is a group in the class $\cC$. 
  If $X$ is a $2$-dimensional hyperbolic $P$-manifold, a {\it connected subsurface in $X$} is the union of a surface $X_v$ with $v \in V_2$ together with $\{ X_e \times [-1,1] \, | \, \text{$e$ is incident to $v$}\}$, a set of annuli. (This subsurface is homeomorphic to $X_v$ and its boundary components are {\it branching curves} on $X$.) A {\it subsurface in $X$} is a nonempty finite union of connected subsurfaces in $X$.
  \end{defn}

  \begin{remark}
   If $X = X(\Lambda)$ is a $2$-dimensional hyperbolic $P$-manifold, then the JSJ graph of the geometric amalgam of free groups $\pi_1(X)$ is the (unoriented) bipartite graph $\gL$. For details, see \cite[Section 4.1]{malone}.
  \end{remark}

  \begin{remark}
    ``$P$-manifold'' is short for ``piecewise-manifold.''
   Lafont \cite{lafont} refers to $2$-dimensional hyperbolic $P$-manifolds as {\it simple, thick, $2$-dimensional hyperbolic $P$-manifolds}; we omit the extra adjectives for ease of exposition. In \cite{danistarkthomas}, these spaces are referred to as {\it surface amalgams}.
  \end{remark}

\section{Degree refinement and related graphs} \label{sec:deg_ref}

  Angluin \cite[Section 6]{angluin} proved two finite graphs have isomorphic universal covers if and only if the graphs have equivalent degree refinements, which is a matrix defined below; see also Leighton~\cite{leighton}. Malone \cite{malone} extended this work by defining the degree refinement for a group $G \in\cC$ and proving this matrix encodes the isomorphism type of the JSJ tree of $G$. Examples of the definitions given in this section appear in Figure~\ref{figure:ex_3}.

            \begin{figure}
      \begin{overpic}[scale=.8, tics=5]{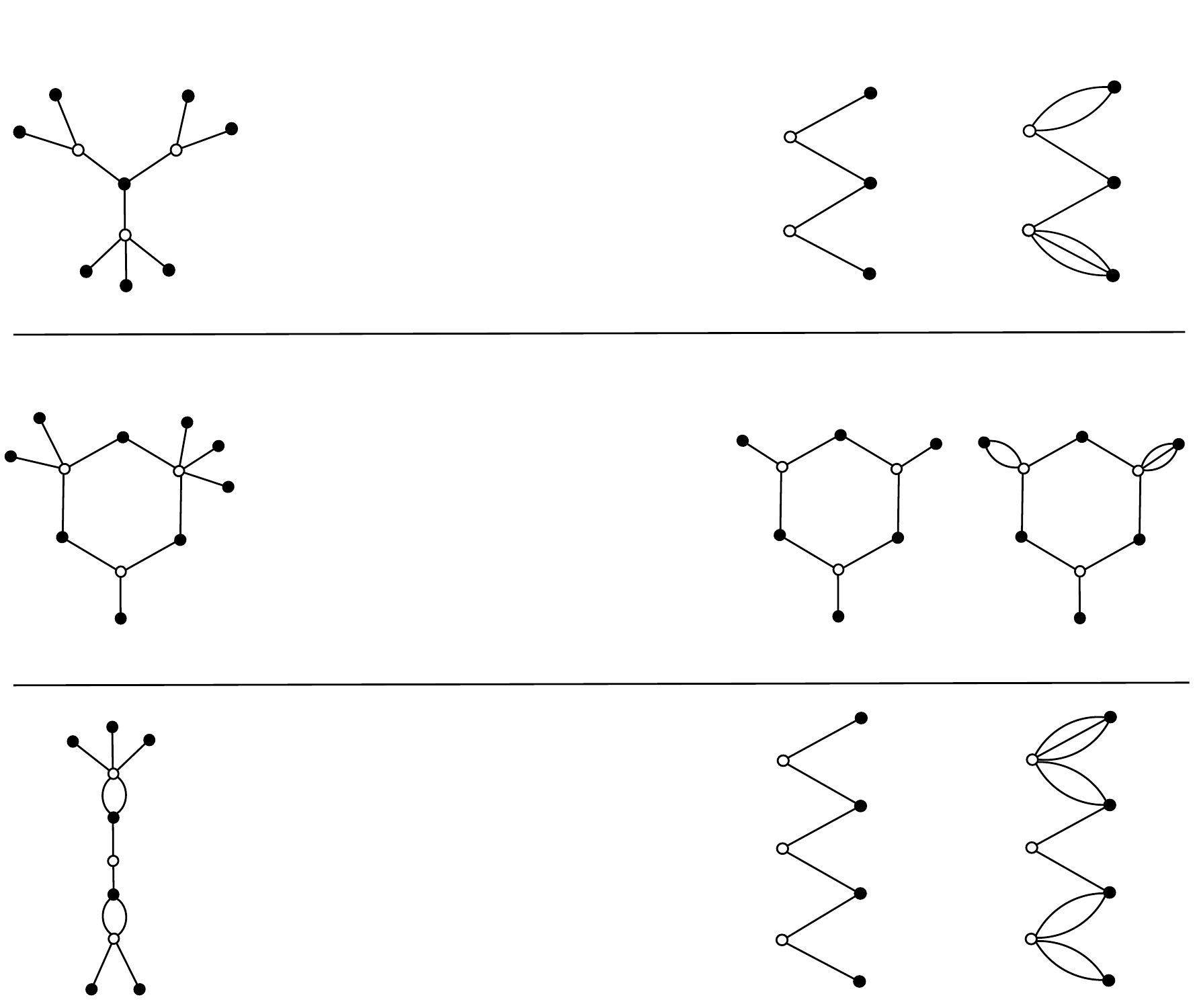}
      \put(2,83.5){JSJ decomposition}
      \put(7,80.5){graph:}
      \put(32,82.5){Degree refinement:}
      \put(65,83.5){Graph of}
      \put(66,80.5){blocks:}
      \put(84.5,83.5){Augmented}
      \put(83,80.5){graph of blocks:}
      \put(31,68){\Small{$\left(\begin{array}{cc|ccc}
		      0 & 0 & 1 & 2 & 0 \\
		      0 & 0 & 1 & 0 & 3\\
		      \hline
		      \infty & \infty & 0 & 0 & 0 \\
		      \infty & 0 & 0 & 0 & 0 \\
		      0 & \infty & 0 & 0 & 0 \\
                   \end{array}\right)$}}
        \put(0,75){\Small{$y_{21}$}}
        \put(4,78){\Small{$y_{22}$}}
        \put(14,78){\Small{$y_{23}$}}
        \put(19,75){\Small{$y_{24}$}}
        \put(10,71){\Small{$y_{1}$}}
        \put(4,60){\Small{$y_{31}$}}
        \put(9,58.5){\Small{$y_{32}$}}
        \put(15,60){\Small{$y_{33}$}}
        \put(4,70){\Small{$x_{11}$}}
        \put(15,70){\Small{$x_{12}$}}
        \put(11.5,65){\Small{$x_{2}$}}
        \put(34,76){\Small{$T_1$}}
        \put(38,76){\Small{$T_2$}}
        \put(42,76){\Small{$F_1$}}
        \put(45.5,76){\Small{$F_2$}}
        \put(49,76){\Small{$F_3$}}
        \put(29,73){\Small{$T_1$}}
        \put(29,70.5){\Small{$T_2$}}
        \put(29,67.5){\Small{$F_1$}}
        \put(29,65){\Small{$F_2$}}
        \put(29,62.5){\Small{$F_3$}}
        \put(63,72.5){\Small{$t_2$}}
        \put(63,64.5){\Small{$t_1$}}
        \put(74,76){\Small{$f_2$}}
        \put(74,68.5){\Small{$f_1$}}
        \put(74,61){\Small{$f_3$}}
        \put(83.5,72.5){\Small{$t_2$}}
        \put(83.5,64.5){\Small{$t_1$}}
        \put(94.5,76){\Small{$f_2$}}
        \put(94.5,68.5){\Small{$f_1$}}
        \put(94.5,61){\Small{$f_3$}}
        \put(24,40){\footnotesize{$\left(\begin{array}{ccc|cccccc}
		      0 & 0 & 0 & 1 & 0 & 0 & 1 & 0 & 1 \\
		      0 & 0 & 0 & 0 & 2 & 0 & 1 & 1 & 0 \\
		      0 & 0 & 0 & 0 & 0 & 3 & 0 & 1 & 1 \\
		      \hline
		      \infty & 0 & 0 & 0 & 0 & 0 & 0 & 0 & 0 \\
		      0 & \infty & 0 & 0 & 0 & 0 & 0 & 0 & 0 \\
		      0 & 0 & \infty & 0 & 0 & 0 & 0 & 0 & 0 \\
		      \infty & \infty & 0 & 0 & 0 & 0 & 0 & 0 & 0 \\
		      0 & \infty & \infty & 0 & 0 & 0 & 0 & 0 & 0 \\
		      \infty & 0 & \infty & 0 & 0 & 0 & 0 & 0 & 0 \\
                   \end{array}\right)$}} 
        \put(-1,47.5){\Small{$y_{21}$}}
        \put(2,51){\Small{$y_{22}$}}
        \put(14,50.5){\Small{$y_{31}$}}
        \put(18,48.5){\Small{$y_{32}$}}
        \put(18,41.5){\Small{$y_{33}$}}
        \put(9,30.5){\Small{$y_{1}$}}
        \put(2,39){\Small{$y_{4}$}}
        \put(9,49){\Small{$y_{5}$}}
        \put(16,39){\Small{$y_{6}$}}
        \put(9,38){\Small{$x_{1}$}}
        \put(6.5,44){\Small{$x_{2}$}}
        \put(12,44){\Small{$x_{3}$}}
        \put(27,53){\Small{$T_1$}}
        \put(31.5,53){\Small{$T_2$}}
        \put(35.8,53){\Small{$T_3$}}
        \put(40,53){\Small{$F_1$}}
        \put(43.5,53){\Small{$F_2$}}
        \put(47,53){\Small{$F_3$}}
        \put(50.5,53){\Small{$F_4$}}
        \put(53.8,53){\Small{$F_5$}}
        \put(57,53){\Small{$F_6$}}
        \put(22,51){\Small{$T_1$}}
        \put(22,48){\Small{$T_2$}}
        \put(22,45.5){\Small{$T_3$}}
        \put(22,42.5){\Small{$F_1$}}
        \put(22,40){\Small{$F_2$}}
        \put(22,37.5){\Small{$F_3$}}
        \put(22,34.7){\Small{$F_4$}}
        \put(22,32){\Small{$F_5$}}
        \put(22,29.5){\Small{$F_6$}}
        \put(69,38){\Small{$t_1$}}
        \put(66,44){\Small{$t_2$}}
        \put(73,44){\Small{$t_3$}}
        \put(69,30){\Small{$f_1$}}
        \put(62,49){\Small{$f_2$}}
        \put(76,49){\Small{$f_3$}}
        \put(63,37){\Small{$f_4$}}
        \put(69,49.5){\Small{$f_5$}}
        \put(75,37){\Small{$f_6$}}
        \put(89,38){\Small{$t_1$}}
        \put(86,43){\Small{$t_2$}}
        \put(93,43){\Small{$t_3$}}
        \put(89,30){\Small{$f_1$}}
        \put(82,49){\Small{$f_2$}}
        \put(97.5,49){\Small{$f_3$}}
        \put(83.5,37){\Small{$f_4$}}
        \put(89.5,49.5){\Small{$f_5$}}
        \put(95,37){\Small{$f_6$}}
        \put(28,11){\Small{$\left(\begin{array}{ccc|cccc}
		      0 & 0 & 0 & 2 & 0 & 2 & 0 \\
		      0 & 0 & 0 & 0 & 3 & 0 & 2 \\
		      0 & 0 & 0 & 0 & 0 & 1 & 1 \\
		      \hline
		      \infty & 0 & 0 & 0 & 0 & 0 & 0 \\
		      0 & \infty & 0 & 0 & 0 & 0 & 0 \\
		      \infty & 0 & \infty & 0 & 0 & 0 & 0 \\
		      0 & \infty & \infty & 0 & 0 & 0 & 0 \\
                   \end{array}\right)$}}
        \put(4,24){\Small{$y_{21}$}}
        \put(8,25){\Small{$y_{22}$}}
        \put(12,24){\Small{$y_{23}$}}
        \put(4,1.5){\Small{$y_{11}$}}
        \put(13,1.5){\Small{$y_{12}$}}
        \put(6.5,9){\Small{$y_{3}$}}
        \put(6.5,15.5){\Small{$y_{4}$}}
        \put(11,5){\Small{$x_{1}$}}
        \put(11,19){\Small{$x_{2}$}}
        \put(11,12){\Small{$x_{3}$}}
        \put(31.2,21.5){\Small{$T_1$}}
        \put(35.7,21.5){\Small{$T_2$}}
        \put(40,21.5){\Small{$T_3$}}
        \put(44,21.5){\Small{$F_1$}}
        \put(47.5,21.5){\Small{$F_2$}}
        \put(51,21.5){\Small{$F_3$}}
        \put(54.5,21.5){\Small{$F_4$}}
        \put(26,19){\Small{$T_1$}}
        \put(26,16.5){\Small{$T_2$}}
        \put(26,13.5){\Small{$T_3$}}
        \put(26,11){\Small{$F_1$}}
        \put(26,8.3){\Small{$F_2$}}
        \put(26,5.7){\Small{$F_3$}}
        \put(26,3){\Small{$F_4$}}
        \put(62,5){\Small{$t_1$}}
        \put(62,20){\Small{$t_2$}}
        \put(62,13){\Small{$t_3$}}
        \put(73.5,2){\Small{$f_1$}}
        \put(73.5,24){\Small{$f_2$}}
        \put(73.5,16.5){\Small{$f_4$}}
        \put(73.5,9){\Small{$f_3$}}
        \put(83,5){\Small{$t_1$}}
        \put(83,20){\Small{$t_2$}}
        \put(83,13){\Small{$t_3$}}
        \put(94,2){\Small{$f_1$}}
        \put(94,24){\Small{$f_2$}}
        \put(94,16.5){\Small{$f_4$}}
        \put(94,9){\Small{$f_3$}}
      \end{overpic}
	\caption{{\small Three examples. White vertices correspond to $2$-ended vertex groups, and black vertices correspond to maximal hanging Fuchsian vertex groups. Vertices $x_i, x_{ij}$ are in degree partition block $T_i$, and vertices $y_i, y_{ij}$ are in degree partition block $F_i$. Any group in $\cC$ with a JSJ decomposition graph of the lower two types is not quasi-isometric to any group generated by finite-order elements.  }}
      \label{figure:ex_3}
     \end{figure}
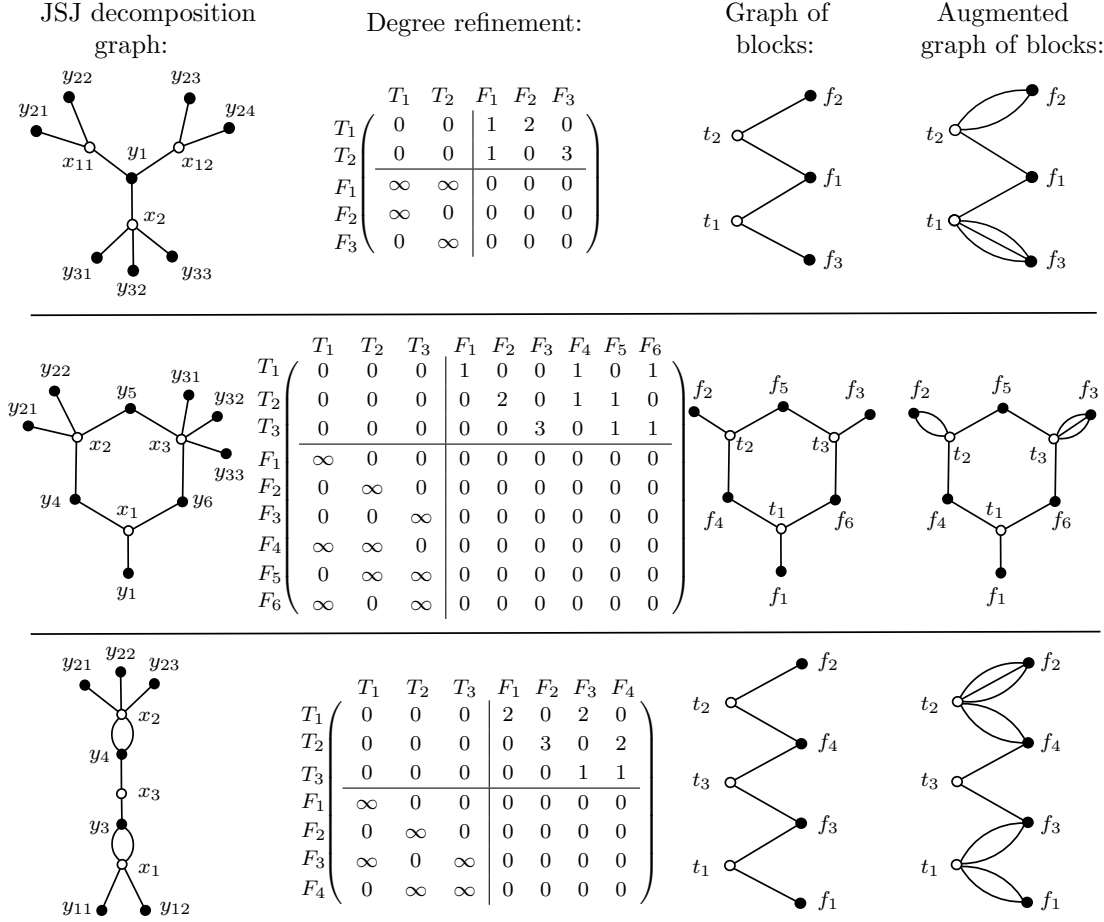
    
  \begin{defn} \label{def:deg_ref} 
    The {\it degree partition} of a graph $\Lambda$ is a partition of the vertices of $\Lambda$ into the minimum number of blocks $M_1, \ldots, M_n$ such that there exist constants $m_{ij}$ such that for each $i,j$ with $1 \leq i\leq n$, $1 \leq j \leq n$, each vertex in $M_i$ is connected via $m_{ij}$ edges to $M_j$. The {\it degree refinement} of $\gL$ is the $n \times n$ matrix $M = (m_{ij})$.  
  \end{defn}
    
  \begin{defn} \label{def:deg_ref_equiv}
   Two degree refinements $M$ and $M'$ are {\it equivalent} if they have the same size and there exists a permutation matrix $P$ so that $M' = PMP^T$. 
  \end{defn}
  
  The permutation matrix in the above definition accounts for possibly relabeling the blocks of the degree partition. There is a correspondence between isomorphism types of trees and equivalence classes of matrices  given as follows. 
  
  \begin{thm} \cite[Section 2]{leighton} \cite[Theorem 2.32]{malone}
   To each matrix $M$ with entries in $\N \cup \{\infty\}$ there is a unique tree $T$ up to graph isomorphism such that $M$ is the degree refinement of $T$. Conversely, to each tree $T$ there is a unique matrix $M$ with entries in $\N \cup \{\infty\}$ up to the equivalence defined in Definition~\ref{def:deg_ref_equiv} so that $M$ is the degree refinement of $T$. 
  \end{thm}
  
  \begin{defn} \label{def:deg_ref_G}
    If $G \in \cC$, then the {\it degree refinement of $G$} is the degree refinement of the JSJ tree of~$G$. Alternatively, the degree refinement of $G$ can be constructed from the JSJ graph for $G$ as follows, which was shown by Malone \cite[Section 2.5]{malone}, extending \cite[Section 2]{leighton}. 
  
    Suppose $G$ has JSJ graph $\Lambda$ with vertex set $V_1 \sqcup V_2$ where each vertex group $G_u$ for $u \in V_1$ is two-ended, and each vertex group $G_v$ for $v \in V_2$ is maximal hanging Fuchsian. 
    
    Let $r \in V(\Lambda)$. The {\it augmented valance of $r$} is the valance of any lift of $r$ in the JSJ tree. More specifically, suppose $s \in V(\Lambda)$. 
    Let $\iota(r,s) = \infty$ if $r \in V_2$ and $r$ is adjacent to $s$. Let $\iota(r,s) = k$ if $r \in V_1$ and $r$ is adjacent to $s$ via $k$ edges. Let $\iota(r,s) = 0$ otherwise. The augmented valance of the vertex $r$ is equal to $\displaystyle \sum_{s \in V(\gL)} \iota(r,s)$.
    
    Perform the following steps to compute the degree refinement. (Step 1) Partition the vertices of $V(\gL)$ into blocks according to their augmented valance. (Step 2) Refine the partition so that two vertices $r, r'$ remain in the same block $M_i$ if and only if for all $j \neq i$,  $\displaystyle \sum_{s\in M_j} \iota(r, s) = \sum_{s \in M_j} \iota(r',s)$. (Step~3) Repeat Step 2 recursively until no further partitioning is possible to obtain the {\it degree partition of $\gL$.} The degree refinement is the matrix $M = (m_{ij})$ where $m_{ij} = \iota(r_i,r_j)$ for $r_i \in M_i$ and $r_j \in M_j$. The process is finite since $V(\gL)$ is finite. 
    \end{defn}
    
   Malone~\cite{malone} proved the following theorem for geometric amalgams of free groups using techniques of Behrstock--Neumann~\cite{behrstockneumann}. Cashen--Martin~\cite{cashenmartin} proved the remaining cases.
    
    \begin{thm} \cite[Theorem 4.14]{malone} \cite[Theorem 4.9]{cashenmartin} \label{thm:qi_same_deg_ref}
    Let $G, G' \in \cC$. The groups $G$ and $G'$ are quasi-isometric if and only if the degree refinement of $G$ is equivalent to the degree refinement of $G'$.
   \end{thm} 
    
   Bipartite graphs considered in this paper arise as a JSJ graph of a group in $\cC$, which leads to the following definition of the degree refinement for a bipartite graph. We caution the reader that this is not the same as the usual notion degree refinement for the graph (without a specified bipartite structure). 

  \begin{defn} \label{def:bipartite}
    If $\Lambda$ is a bipartite graph with $V(\gL) = V_1 \sqcup V_2$, define the {\it degree refinement of $\gL$} to be the degree refinement of a group with JSJ graph $\Lambda$ as defined in Definition \ref{def:deg_ref_G}. Similarly, define the {\it degree partition of $\gL$} to be the degree partition of a group with JSJ graph $\Lambda$ as defined in Definition~\ref{def:deg_ref_G}.
    
   Suppose in the degree partition of $\gL$, the vertices in $V_1$ are contained in blocks $T_1, \ldots, T_n$ and the vertices in $V_2$ are in blocks $F_1, \ldots F_m$. (We use ``T'' for $t$wo-ended and ``F'' for hanging $F$uchsian.)
   Let $M$ be the degree refinement defined above. Let $n_{ij}$ be the entry in the degree refinement corresponding to the blocks $T_i$ and $F_j$. The {\it graph of blocks of $M$}, denoted $\G_{\cB}$, has vertex set $\{t_1, \ldots, t_n, f_1, \ldots, f_m\}$ and an edge $\{t_i, f_j\}$ if and only if $n_{ij}>0$. The {\it augmented graph of blocks of $M$}, denoted $\G_0$, has vertex set $\{t_1, \ldots, t_n, f_1, \ldots, f_m\}$ and $n_{ij}$ edges from $t_i$ to $n_j$ for $1 \leq i \leq n$, $1 \leq j \leq m$. 
   
   If $G \in \cC$ with JSJ graph $\Lambda$ as above, then the {\it graph of blocks of $G$} is the graph of blocks of the degree refinement of $G$; the {\it augmented graph of blocks of $G$} is the augmented graph of blocks of the degree refinement of $G$. Similarly, if $\Lambda$ is a bipartite graph, then the {\it graph of blocks of $\Lambda$} is the graph of blocks of any group $G \in \cC$ with JSJ graph $\gL$; the {\it augmented graph of blocks of $\gL$} is the augmented graph of blocks of $G$.
  \end{defn}
  
  \begin{remark}
   The bipartite graph $\G_0$ defined in Definition~\ref{def:bipartite} has degree refinement $M$. 
  \end{remark}

 
  The next corollary follows from Theorem~\ref{thm:qi_same_deg_ref}.
 
   \begin{cor} \label{cor_same_block_gr}
    Let $G, G' \in \cC$, let $\G_{\cB}$ and $\G_{\cB}'$ denote the graphs of blocks for $G$ and $G'$, respectively, and let $\G_0$ and $\G_0'$ denote the augmented graphs of blocks of $G$ and $G'$, respectively. If $G$ and $G'$ are quasi-isometric, then $\G_{\cB} \cong \G_{\cB}'$ and $\G_0 \cong \G_0'$. 
   \end{cor}
 
\section{Characterization up to quasi-isometry} \label{sec:char_qi}

\subsection{Obstructions} \label{subsec:obstructions}

  In this section, we prove that Conditions (1)-(3) of Theorem~\ref{thm_sec1_qi_char} imply Condition~(4). 
  The first lemma generalizes \cite[Example 8.1]{danistarkthomas}.  An example of a group that does not satisfy the hypothesis of Lemma~\ref{lemma:no_cycle} is given in the middle example in Figure \ref{figure:ex_3}; an example of a group which does not satisfy the hypothesis of Lemma~\ref{lemma:bad_double} is given in the bottom example in Figure~\ref{figure:ex_3}.
  
  \begin{lemma} \label{lemma:no_cycle}
   Suppose $G \in \cC$ has JSJ graph a tree $T$. Then the graph of blocks of $G$ is a tree. 
  \end{lemma}
  \begin{proof}
    Suppose $G \in \cC$ has JSJ graph a tree $T$, and let $\G_{\cB}$ be the graph of blocks of $G$. Suppose towards a contradiction that $\G_{\cB}$ is not a tree. Then $\G_{\cB}$ contains an embedded cycle $\gamma$. Without loss of generality, suppose the vertices in the cycle $\gamma$ are labeled $\{t_1, f_1, t_2, \ldots, t_k, f_k\}$ with $t_i$ adjacent to $f_i$ and $f_{i-1}$, with indices taken mod $k$, and likewise for $f_i$. Let $T_1, \ldots T_k, F_1, \ldots F_k$ be the corresponding blocks in the degree partition of $T$. Consider the subgraph of $T$ induced by the vertices in $\cup_{i=1}^k T_i \cup F_i$. Choose a connected component $C$ of the subgraph; then $C$ must be a finite tree. However, each vertex in $C$ has valance at least two. That is, if $v_i \in F_i$, then $v_i$ is adjacent to some vertices $u_i \in T_i$ and $u_{i+1} \in T_{i+1}$, with indices taken mod $k$. Since the cycle $\gamma$ is embedded, $T_i \cap T_{i+1} = \emptyset$; so, $u_i \neq u_{i+1}$. Similarly, if $v_i \in T_i$, $v_i$ has valance at least two, a contradiction.     
  \end{proof}
  
  \begin{defn} \label{defn:no2cycles}
   Let $G \in \cC$ with degree refinement and related graphs as defined in Definition~\ref{def:deg_ref_G}. We say {\it the augmented graph of blocks of $G$ has no $2$-cycles at even distance bounded by Type~I vertices} if whenever $t_1,f_1,t_2, \ldots, f_{k-1},t_k$ is an embedded path in $\G_0$ with $t_i \in V_1$ and $f_j \in V_2$, then if $n_{11}>1$, then $n_{ij}>1$ only if $i=j$.
  \end{defn}

  \begin{lemma} \label{lemma:bad_double}
    Suppose $G \in \cC$ has JSJ graph a tree. Then the augmented graph of blocks of $G$ has no $2$-cycles at even distance bounded by Type~I vertices.
  \end{lemma}
  \begin{proof}
   Suppose $G \in \cC$ has JSJ graph a tree. 
   By Lemma~\ref{lemma:no_cycle}, the graph of blocks of $G$ is a tree. Thus, in the notation of Definition~\ref{defn:no2cycles}, $n_{ij} \geq 1$ only if $i=j$ or $i = j+1$.
   Suppose towards a contradiction that $n_{\ell,\ell-1} >1$ for some $\ell$ with $1 \leq \ell \leq k$. Let $T_1, \ldots, T_{\ell} \subset V_1$ and $F_1, \ldots, F_{\ell-1}\subset V_2$ be the corresponding blocks in the degree partition of $T$. Consider the subgraph of $T$ induced by the vertices in $(\cup_{i=1}^{\ell} T_i) \cup (\cup_{i=1}^{\ell-1} F_i)$. Choose a connected component $C$ of this subgraph; then $C$ must be a finite tree. However, each vertex in $C$ has valance at least two. That is, for $v_i\in F_i$, $1 \leq i \leq \ell-1$, $v_i$ is adjacent to some vertices $u_i \in T_i$ and $u_{i+1}\in T_{i+1}$ with $u_i \neq u_{i+1}$ since $T_i \cap T_{i+1} = \emptyset$. Likewise, $v_i \in T_i$ has valance two for $2 \leq i \leq \ell-1$. Since $n_{11}, n_{\ell,\ell-1}>1$, if $v_i \in T_i$ for $i=1, \ell$, then $v_i$ has valance at least two. Thus, each vertex of $C$ has valence at least two, a contradiction since $C$ is a finite tree. 
  \end{proof}
  
  \begin{remark}
   In the notation of Definition~\ref{defn:no2cycles}, a group in $\cC$ may have $2$-cycles at even distance bounded by Type II vertices. An example is given in the top of Figure~\ref{figure:ex_3}.
  \end{remark}

\subsection{Construction} \label{subsec:construction}

  \begin{outline} \label{outline_const}
   Let $M$ be a degree refinement of a group in $\cC$ that satisfies Conditions (M1) and (M2) of Theorem~\ref{thm_sec1_qi_char}. Let $\G_{\cB}$ be the graph of blocks of $M$, and let $\G_0$ be the augmented graph of blocks of~$M$. We will describe a finite process to construct a finite bipartite tree with degree refinement $M$ (as in Definition \ref{def:bipartite}). The bipartite graph $\G_0$ has degree refinement $M$, but, in general, $\G_0$ is not a tree. We will perform a finite series of moves on $\G_0$ to produce a finite tree. The moves on the graph $\G_0$ recursively unwrap the cycles of length two in $\G_0$ so that each move preserves the degree refinement. 
  \end{outline}

  An image of the following definition appears in Figure~\ref{figure:split_vertex}.
  
     \begin{figure}
      \begin{overpic}[scale=.8,  tics=5]{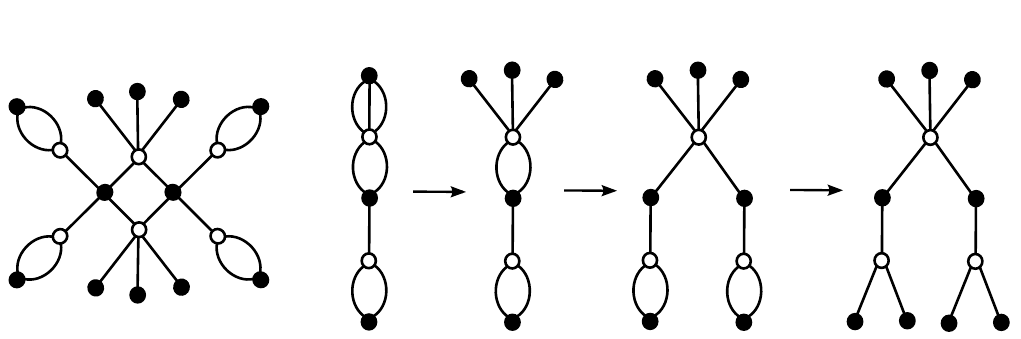}
      \put(12,30){$\Lambda$}
      \put(34,30){$\G_0$}
      \put(90,30){$T$}
        \end{overpic}
	\caption{{\small Example: $\Lambda$ is the JSJ graph of a group in $\cC$ that satisfies conditions (M1) and (M2) of Theorem~\ref{thm_sec1_qi_char}, and $\G_0$ is the augmented graph of blocks of $\Lambda$. Between the graphs $\Gamma_0$ and $T$, the {\it split a vertex} move was performed four times to produce a finite tree with the same degree refinement as $\Lambda$ and $\G_0$.   }}
      \label{figure:split_vertex}
     \end{figure}

\begin{defn} \label{def:split} (Split a MHF vertex.)
  Let $\Lambda$ be a bipartite graph with $V(\gL) = V_1 \sqcup V_2$ as defined in Definition \ref{def:bipartite}. Let $t \in V_1$ and $f \in V_2$, and suppose $\Lambda$ has $r>0$ edges $e_1, \ldots, e_r$ with endpoints $\{t,f\}$. Let $m_i \in \Lambda$ be the midpoint of the edge $e_i$, and suppose that $\cup_{i=1}^r m_i$ separates $\gL$ into two components. Then $\Lambda - \cup_{i=1}^r \Int(e_i)$ has two components; let $C \subset \Lambda$ be the component containing $f$, and let $C' \subset \Lambda$ be the component containing $t$.   
   (So, $\gL = C \cup C' \cup \left( \cup_{i=1}^r e_i \right)$.)   
  Define $\gL'$ to be the following finite graph, which is obtained by {\it splitting~$f$} into $r$ vertices. For $1 \leq i \leq r$, let $C_i$ be a graph isomorphic to $C$, and let $\phi_i:C \rightarrow C_i$ be a graph isomorphism. Let $f_i = \phi_i(f)$. Let $\gL'$ be the graph formed by the union of $C'$, $\cup_{i=1}^r C_i$, and $r$ edges $e_1', \ldots, e_r'$, where $e_i'$ has one endpoint $t \in C'$ and the other endpoint $f_i \in C_i$. Let $p:\gL' \rightarrow \gL$ be the projection that is the identity on $C'$, maps $e_i'$ to $e_i$ and $C_i$ to $C$ by an isomorphism. 
\end{defn}

  \begin{lemma} \label{lemma:deg_ref_pres}
   The degree refinement of $\gL$ is equivalent to the degree refinement of $\gL'$, where $\gL$ and $\gL'$ are the graphs defined in Definition~\ref{def:split}. 
  \end{lemma}
  \begin{proof}
   Suppose $T_1, \ldots, T_n, F_1, \ldots, F_m$ are the blocks in the degree partition of $\gL$. As in Definition~\ref{def:bipartite}, there exist constants $n_{ij}$ with $1 \leq i \leq n$ and $1 \leq j \leq m$ so that each vertex in $T_i$ is adjacent via $n_{ij}$ edges to $F_j$ and if $n_{ij}>0$, each vertex in $F_j$ is adjacent via at least one edge to $T_i$. There exists a partition of the vertices of $\gL'$ into blocks $T_1', \ldots, T_n', F_1', \ldots, F_m'$ where $T_i' =  p^{-1}(T_i)$ and $F_i' = p^{-1}(F_i)$, where $p$ is the projection map in Definition~\ref{def:split}. By construction, since adjacencies are affected only between $f$ and $t$, each vertex in $T_i'$ is adjacent via $n_{ij}$ edges to $F_j'$ and if $n_{ij}>0$, each vertex in $F_j'$ is adjacent via at least one edge to $T_i'$.
  \end{proof}

  \begin{construction}\label{const:new_graph}
   Let $\G_0$ be the augmented graph of blocks of $M$ as in Outline~\ref{outline_const}. Suppose $V(\G_0) = V_1 \sqcup V_2$ as in Definition~\ref{def:deg_ref_G}. The following construction produces a finite tree $T$ with degree refinement $M$. Since $\Gamma_0$ has degree refinement $M$, if $\Gamma_0$ is a tree, no additional moves are necessary. Otherwise, suppose $\G_0$ is not a tree. Let $t \in V_1$ so that there exists $f \in V_2$ and $r>1$ edges $e_1, \ldots, e_r$ connecting $t$ and $f$. By Condition~(M1), removing the interiors of the edges in this collection disconnects the graph $\Gamma_0$. Let $C \subset \Lambda - \bigcup_{i=1}^r \Int(e_i)$ be the component containing $f$.  Let $D = C \cup \bigcup_{i=1}^r e_i$. 
   
   Define a height function $h:V(D) \rightarrow \N$ by $h(v) = d(v,t)$. Then $h(v) \in 2\Z$ if and only if $v \in V_1$. Furthermore, if there exists $t'\in D \cap V_1$ and $f' \in D \cap V_2$ and $r'>1$ edges connecting $t'$ and $f'$, then $h(f') = h(t') +1$; otherwise, Condition (M2) would be violated. 
   
   Perform a series of moves recursively to split vertices (as in Definition \ref{def:split}) of $D$ at height $1, 3, 5, \ldots$, and so on. Let $D^k$ denote the graph obtained after vertices at height $k$ have been split. As in Definition~\ref{def:split}, there are projections $D^k \rightarrow D^{k-2}$ for $k\geq 3$ and $D_1 \rightarrow D$. Denote the composition of these maps as the projection $p_k:D^k \rightarrow D$ for $k \geq 1$. In an abuse of notation, we use $t$ to denote the unique vertex in $p_k^{-1}(t) \in D^k$. 
   Let $h_k:D^k \rightarrow \N$ be given by $h_k(v) = d(v,t)$. Then $h_k(v) = h(p_k(v))$ for all $v \in D^k$. Let $D^k_{\leq k}$ be the subgraph of $D^k$ induced by vertices with height $\leq k$. Then $D^k_{\leq k}$ is a finite tree. Since $D$ is a finite graph, there exists $N \in \N$ so that $h(v) <N$ for all $v \in V(D)$. Thus, after finitely many moves, the resulting graph $D^N$ is a tree. 
   
   The moves on $\G_0$ are the identity on $\G_0 \backslash D$, a graph which may still contain (finitely many) embedded cycles of length two. For each set of additional cycles, the above procedure can be performed. By Condition (M1) $\G_{\cB}$ is a tree, so these cycles of length $2$ are the only embedded cycles in $\G_0$. Hence, after finitely many steps, the resulting graph is a tree. In addition, by Lemma \ref{lemma:deg_ref_pres}, the resulting graph has the same degree refinement as $\G_0$, as desired.    
  \end{construction}

  \begin{prop} \label{prop:make_tree_gp}
   If $G \in \cC$ satisfies Conditions (M1) and (M2) of Theorem~\ref{thm_sec1_qi_char}, then $G$ is quasi-isometric to a group with JSJ graph a tree.
  \end{prop}
  \begin{proof}
   Let $M$ be the degree refinement of $G$, and let $\G_0$ be the augmented graph of blocks of $M$. Then $\G_0$ has degree refinement $M$, and by Construction~\ref{const:new_graph}, there exists a finite tree $T$ with degree refinement~$M$. Let $G' \in \cC$ be any group with JSJ graph $T$. By Theorem~\ref{thm:qi_same_deg_ref}, the groups $G$ and $G'$ are quasi-isometric. 
  \end{proof}
  
  \subsection{Characterization} \label{sec:char}

  We collect the above conditions and constructions to prove one of the main theorems of the paper.

   \begin{thm} \label{thm_qi_char_later_sec}
  Let $G \in \cC$. The following are equivalent.
  \begin{enumerate}
     \item The group $G$ is quasi-isometric to a right-angled Coxeter group.
    \item The group $G$ is quasi-isometric to a group generated by finite-order elements.
    \item The group $G$ is quasi-isometric to a group with JSJ graph a tree.
    \item The degree refinement of $G$ satisfies the two conditions:
      \begin{itemize}
	\item[(M1)] The graph of blocks of $G$ is a tree.
	\item[(M2)] The augmented graph of blocks of $G$ has no $2$-cycles at even distance bounded by Type I vertices. 
      \end{itemize}
  \end{enumerate}
 \end{thm}
 
 \begin{proof}
  Let $G \in \cC$. We first show (3) and (4) are equivalent. To prove (3) implies (4), let $G' \in \cC$ be a group with JSJ graph a tree which is quasi-isometric to $G$. Let $\G_{\cB}$ and $\G_{\cB}'$ be the graphs of blocks of $G$ and $G'$, respectively, and let $\G_0$ and $\G_0'$ be the augmented graphs of blocks of $G$ and $G'$, respectively. By Corollary~\ref{cor_same_block_gr}, $\G_{\cB} \cong \G_{\cB}'$ and $\G_0 \cong \G_0'$. Therefore, by Lemma \ref{lemma:no_cycle} and Lemma \ref{lemma:bad_double}, Conditions (M1) and (M2) hold. 
  By Proposition~\ref{prop:make_tree_gp}, (4) implies (3).
  
  Clearly, (1) implies (2) and (2) implies (3). Suppose $G \in \cC$ is quasi-isometric to a group with JSJ graph a tree. Then $G$ is quasi-isometric to a geometric amalgam of free groups with JSJ graph a tree. Thus, by \cite[Theorem~1.16]{danistarkthomas}, $G$ is quasi-isometric to a right-angled Coxeter group, so (3) implies (1), concluding the proof. 
 \end{proof}

  \section{Commensurability classes} \label{sec_comm_classes}

  \subsection{The structure of finite-index subgroups}
  
    The structure of subgroups of a graph of groups is described by Scott--Wall in \cite[Section 3]{scottwall}. In this subsection, we record the facts and constructions relevant to this paper. 
    
    As described in Section~\ref{sec:jsj}, the JSJ decomposition $\cG$ of a group $G \in \cC$ is a splitting of $G$ as the fundamental group of a graph of groups. There is a finite CW-complex $X$ which is a graph of spaces associated to $\cG$ and so that $G \cong \pi_1(X)$. Any finite-index subgroup $H \leq G$ is the fundamental group of a graph of spaces $Y$ which finitely covers $X$. Thus, $H$ splits as a graph of groups. Moreover, the graph of groups splitting of $H$ associated to this graph of spaces $Y$ is the JSJ decomposition of $H$. The details are as follows. 
  
  \begin{prop} \label{prop:sub_struc}
   Let $G \in \cC$ and $H \leq G$ be a finite-index subgroup. Let $\cG$ be the JSJ decomposition of $G$ with underlying graph $\gL$. The subgroup $H$ is the fundamental group of a graph of groups $\cH$ associated to the graph of spaces $Y$ described above. If $\G$ is the underlying graph of $\cH$, then for each $w \in V(\G)$, there exists $v \in V(\gL)$ and $g_w \in G$ so that $H_w = H \cap g_wG_vg_w^{-1}$ and $H_w$ is a finite-index subgroup of $g_wG_vg_w^{-1}$. The graph of groups $\cH$ is the JSJ decomposition of~$H$. 
  \end{prop}
  \begin{proof}
    All statements except the last sentence of the proposition are given in \cite[Section 3]{scottwall}; it remains to show that $\cH$ is the JSJ decomposition of $H$. Indeed, suppose $w \in V(\G)$ so that $H_w = H \cap g_wG_vg_w^{-1}$. By the construction of the JSJ decomposition of $G$ due to Bowditch \cite{bowditch}, the subgroup $g_wG_vg_w^{-1}$ is the stabilizer in $G$ of a distinguished subset $A \subset \p_{\infty} G$ in the visual boundary of $G$ called either a  {\it necklace} or a {\it jump}, depending on whether the group $G_v$ is maximally hanging Fuchsian or $2$-ended, respectively. (See \cite[Section 5]{bowditch}.) Since $H$ is a finite-index subgroup of $G$, the inclusion of $H$ in $G$ induces a homeomorphism from the visual boundary of $H$ to the visual boundary of $G$. Thus, $H_w$ is the stabilizer in $H$ of the subset $A \subset \p_{\infty} H \cong \p_{\infty} G$. Therefore, $H_w$ is a vertex group in the JSJ decomposition of $H$. By the same reasoning, the adjacencies between vertex groups in $\cG$ yield the appropriate adjacencies between vertex groups in $\cH$. Therefore, $\cH$ is the JSJ decomposition of $H$. 
  \end{proof}
 
  \begin{notation} \label{nota:subgroup}
    Suppose $G,G' \in \cC$ are abstractly commensurable. Let $\cG$ and $\cG'$ be the JSJ decompositions of $G$ and $G'$, respectively. Suppose $\cG$ and $\cG'$ have underlying graphs $\Lambda$ and $\gL'$, respectively, with $V(\gL) = V_1 \sqcup V_2$ and $V(\gL') = V_1' \sqcup V_2'$. Suppose each vertex group $G_v$ and $G_{v'}'$ is $2$-ended for $v \in V_1$ and $v'\in V_1'$, respectively; suppose each vertex group $G_v$ and $G_{v'}'$ is maximally hanging Fuchsian for each $v \in V_2$ and $v' \in V_2'$, respectively. Let $H \leq G$ and $H' \leq G'$ be subgroups of finite-index with $H \cong H'$. Let $\cH$ and $\cH'$ be the JSJ decompositions of $H$ and $H'$, respectively. Suppose $\cH$ and $\cH'$ have underlying graphs $\G$ and $\G'$, respectively, with $V(\G) = W_1 \sqcup W_2$ and $V(\G') = W_1' \sqcup W_2'$. Suppose each vertex group $H_w$ and $H_{w'}'$ is $2$-ended for $w \in W_1$ and $w'\in W_1'$, respectively; suppose each vertex group $H_w$ and $H_{w'}'$ is maximally hanging Fuchsian for each $w \in W_2$ and $w' \in W_2'$, respectively.
    
    Suppose in the degree partition of $\gL$ the vertices in $V_1$ are contained in blocks $T_1(G), \ldots, T_n(G)$, and the vertices in $V_2$ are contained in blocks $F_1(G), \ldots, F_m(G)$. All groups in $\{G,G',H,H'\}$ are quasi-isometric. Hence, by Theorem~\ref{thm:qi_same_deg_ref}, in the degree partition of $\gL'$ the vertices in $V_1'$ may be partitioned into blocks $T_1(G'), \ldots, T_n(G')$ and the vertices in $V_2'$ may be partitioned into blocks $F_1(G'), \ldots, F_m(G')$ so that the resulting degree refinement matrix for $\gL'$ is equal to the degree refinement matrix for $\gL$ (without permuting the indices of the $T_i(G')$ and $F_i(G')$). Similarly, assume the degree partition of $\G$ into blocks $T_1(H), \ldots, T_n(H), F_1(H), \ldots, F_m(H)$ and the degree partition of $\G'$ into blocks $T_1(H'), \ldots, T_n(H'), F_1(H'), \ldots, F_m(H')$ also yield degree refinement matrices equal to the degree refinement matrix of $\gL$ and $\gL'$ (without permuting the indices of the blocks). 
  
    The degree partitions of the vertices of the JSJ graphs yield natural partitions of the vertex groups in the JSJ decompositions. We will use the following notation:
      \begin{eqnarray*}
       T_i^G = \{ G_v \, | \, v \in T_i(G)\} &\quad\quad& F_i^G = \{ G_v \, | \, v \in F_i(G)\}\\
       T_i^{G'} = \{ G_{v'}' \, | \, v' \in T_i(G')\} &\quad\quad& F_i^{G'} = \{ G_{v'}' \, | \, v' \in F_i(G')\}\\
       T_i^H = \{ H_w \, | \, w \in T_i(H)\} &\quad\quad& F_i^H = \{ H_w \, | \, w \in F_i(H)\}\\
       T_i^{H'} = \{ H_{w'}' \, | \, w' \in T_i(H')\} &\quad\quad& F_i^{H'} = \{ H_{w'}' \, | \, w' \in F_i(H')\}.\\
      \end{eqnarray*}
    \end{notation}
  
  \begin{lemma} \label{lemma:same_block}
   Let $H \leq G$ be a finite-index subgroup with the notation defined above. Let $w \in V(\G)$ so that there exists $v \in V(\gL)$ and $g_w \in G$ (by Proposition~\ref{prop:sub_struc}) so that $H_w = H \cap g_wG_vg_w^{-1}$. Then, $G_v \in F_i^G$ if and only if $H_w \in F_i^H$. Similarly, $G_v \in T_i^G$ if and only if $H_w \in T_i^H$. 
  \end{lemma}
    \begin{proof}
     Since $H$ is a finite-index subgroup of $G$, the inclusion of $H$ in $G$ induces a homeomorphism from the visual boundary of $H$ to the visual boundary of $G$. Hence, $H$ and $G$ have isomorphic JSJ trees by the construction of the JSJ tree due to Bowditch. Denote this tree by $T$. The subgroup $H_w$ stabilizes the same vertex of $T$ as $g_wG_vg_w^{-1}$, and the conclusion of the lemma follows. 
    \end{proof}

  To study finite covers of a $2$-dimensional hyperbolic $P$-manifold $X$, one often considers the full pre-image of a singular curve on $X$ or a subsurface of $X$. Algebraically, this corresponds to considering the following subset of vertex groups. 

  \begin{defn}
    For $v \in V(\gL)$, let 
     \[\cH_v = \{ H_w \leq H, w \in V(\G) \, \,| \,\, \text{there exists } g_w \in G \text{ so that } H_w = H \cap g_wG_vg_w^{-1}\} .\]  Define $\cH_v'$ similarly. 
  \end{defn}

  We will make use of the following two elementary observations, which follow from the discussion in \cite[Section 3]{scottwall}. 
  
  \begin{lemma} \label{lemma:sum_chi}
   Suppose $H \leq G$ is a subgroup of index $d$. For each $v \in V(\Lambda)$,
    \[ \sum_{H_w \in \cH_v} [G_v : H_w] = d. \] 
  \end{lemma}

  \begin{lemma} \label{lemma:adj_vert}
   If $w,w' \in V(\G)$ are adjacent, $H_w  = H \cap g_wG_vg_w^{-1}$, and $H_{w'} = H \cap g_{w'}G_{v'}g_{w'}^{-1}$ for some $g_w,g_{w'} \in G$ and $v, v' \in V(\gL)$, then $v$ and $v'$ are adjacent. 
  \end{lemma}

    \subsection{Quasi-isometric rigidity does not hold for the subclass of torsion-generated groups} 
    
   \begin{construction} \label{const:not_AC}

      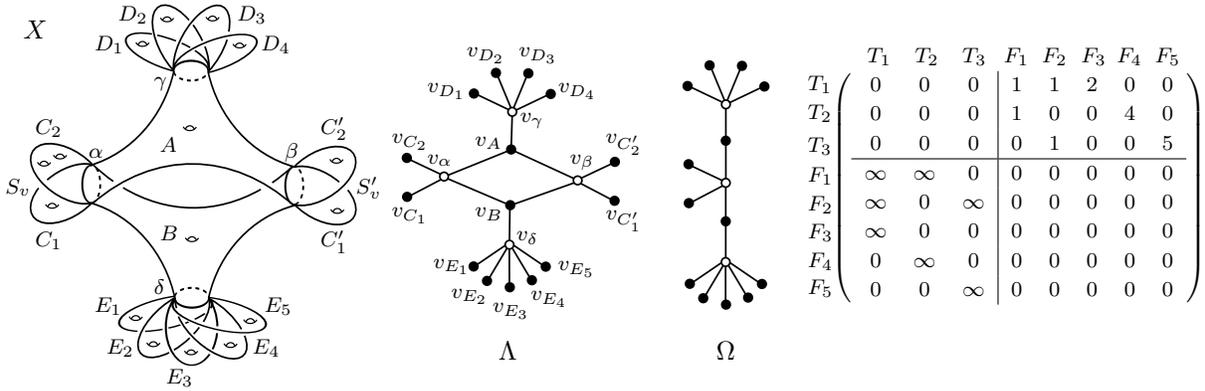
\begin{figure}
      \begin{overpic}[scale=.8,tics=5]{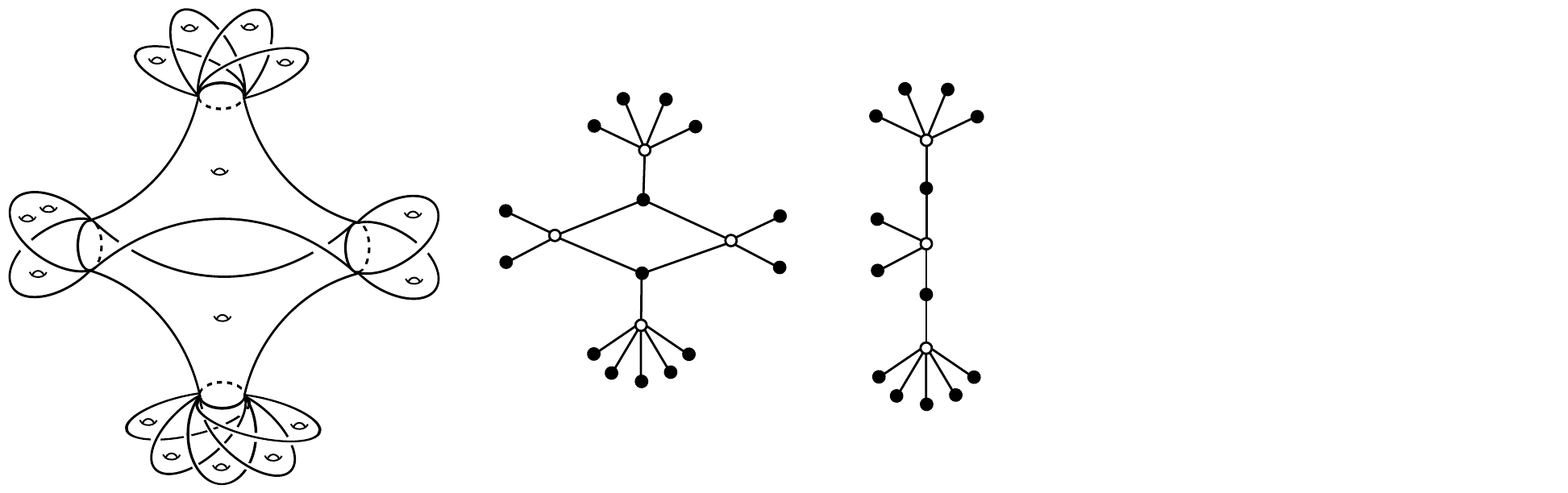}
	\put(68,15){\Small{$\left(\begin{array}{ccc|ccccc}
         0 & 0 & 0 & 1 & 1 & 2 & 0 & 0 \\
         0 & 0 & 0 & 1 & 0 & 0 & 4 & 0 \\
         0 & 0 & 0 & 0 & 1 & 0 & 0 & 5 \\
        \hline
        \infty & \infty & 0 & 0 & 0 & 0 & 0 & 0 \\
         \infty & 0 & \infty & 0 & 0 & 0 & 0 & 0 \\
         \infty & 0 & 0 & 0 & 0 & 0 & 0 & 0 \\
         0 & \infty & 0 & 0 & 0 & 0 & 0 & 0 \\
         0 & 0 & \infty & 0 & 0 & 0 & 0 & 0 \\
        \end{array}\right)$}}
        \put(71,26){\Small{$T_1$}}
        \put(75,26){\Small{$T_2$}}
        \put(79,26){\Small{$T_3$}}
        \put(82.5,26){\Small{$F_1$}}
        \put(85.7,26){\Small{$F_2$}}
        \put(89,26){\Small{$F_3$}}
        \put(92,26){\Small{$F_4$}}
        \put(95.2,26){\Small{$F_5$}}
        \put(66,23.7){\Small{$T_1$}}
        \put(66,21.2){\Small{$T_2$}}
        \put(66,18.5){\Small{$T_3$}}
        \put(66,16){\Small{$F_1$}}
        \put(66,13.5){\Small{$F_2$}}
        \put(66,11.2){\Small{$F_3$}}
        \put(66,8.8){\Small{$F_4$}}
        \put(66,6.5){\Small{$F_5$}}
        \put(40,1){$\Lambda$}
        \put(58.3,1){$\Omega$}
        \put(0,28){$X$}
        \put(1,10.5){\Small{$C_1$}}
        \put(1,20){\Small{$C_2$}}
        \put(25,10.5){\Small{$C_1'$}}
        \put(25,20){\Small{$C_2'$}}
        \put(-1.5,15){\Small{$S_v$}}
        \put(28,15){\Small{$S_v'$}}
        \put(11.5,18.5){\Small{$A$}}
        \put(11.5,11){\Small{$B$}}
        \put(6,27){\Small{$D_1$}}
        \put(8,29.5){\Small{$D_2$}}
        \put(18,29.5){\Small{$D_3$}}
        \put(20,27){\Small{$D_4$}}
        \put(6,5){\Small{$E_1$}}
        \put(7,1.5){\Small{$E_2$}}
        \put(12,-1){\Small{$E_3$}}
        \put(19.3,1.5){\Small{$E_4$}}
        \put(20.3,5){\Small{$E_5$}}
        \put(5.5,18){\Small{$\alpha$}}
        \put(22,18){\Small{$\beta$}}
        \put(11,24){\Small{$\gamma$}}
        \put(11,6.5){\Small{$\delta$}}
        \put(34,17.5){\Small{$v_{\alpha}$}}
        \put(46,17.5){\Small{$v_{\beta}$}}
        \put(41.8,21){\Small{$v_{\gamma}$}}
        \put(41.5,11){\Small{$v_{\delta}$}}
        \put(38,19){\Small{$v_A$}}
        \put(38,13){\Small{$v_B$}}
        \put(31,13){\Small{$v_{C_1}$}}
        \put(31,19.2){\Small{$v_{C_2}$}}
        \put(49,12.8){\Small{$v_{C_1'}$}}
        \put(49,19.2){\Small{$v_{C_2'}$}}
        \put(34,23.5){\Small{$v_{D_1}$}}
        \put(37.3,26.5){\Small{$v_{D_2}$}}
        \put(41.7,26.5){\Small{$v_{D_3}$}}
        \put(45,23.5){\Small{$v_{D_4}$}}
        \put(34.5,8.7){\Small{$v_{E_1}$}}
        \put(36,6.2){\Small{$v_{E_2}$}}
        \put(39.5,5){\Small{$v_{E_3}$}}
        \put(42.7,6){\Small{$v_{E_4}$}}
        \put(45,8.5){\Small{$v_{E_5}$}}
      \end{overpic}
	\caption{{\small The group $\pi_1(X)$ has JSJ decomposition graph $\Lambda$ and degree refinement the matrix shown. The group $\pi_1(X)$ is quasi-isometric to a right-angled Coxeter group with JSJ graph $\Omega$, but $\pi_1(X)$ is not abstractly commensurable to any group generated by finite-order elements as shown in Theorem~\ref{thm:qi_notAC}.}}
      \label{figure:not_AC_tree}
     \end{figure}
     
     Let $G \cong \pi_1(X) \in \cC$ be the fundamental group of the $2$-dimensional hyperbolic $P$-manifold $X$ shown in Figure~\ref{figure:not_AC_tree}. 
    The group $G$ has JSJ decomposition with JSJ graph $\Lambda$. The vertices in Figure~\ref{figure:not_AC_tree} are labeled so that $G_{v_K} \cong \pi_1(K)$, where $K$ is a connected subsurface or branching curve in the space $X$. 
    The JSJ graph for $G$ has the following degree partition.
    \begin{center}
     $\begin{array}{ccccc}
     T_1(G) = \{v_{\alpha}, v_{\beta}\} & \quad \quad & F_1(G) = \{v_A\} & \quad \quad & F_4(G) = \{v_{D_1}, \ldots, v_{D_4}\} \\
     T_2(G) = \{v_{\gamma}\} & \quad \quad & F_2(G) = \{v_B\}& \quad \quad & F_5(G) = \{v_{E_1}, \ldots, v_{E_5}\}  \\
     T_3(G) = \{v_{\delta} \} & \quad \quad & F_3(G) = \{v_{C_1}, v_{C_2}, v_{C_1'}, v_{C_2'}\}  & \quad \quad & \\
     \end{array}$
    \end{center}
    The degree refinement for $G$ is the matrix in Figure~\ref{figure:not_AC_tree}.
   \end{construction}
   
   \begin{thm} \label{thm:qi_notAC}
    Let $G \cong \pi_1(X) \in \cC$ as given in Construction~\ref{const:not_AC}. The following hold.
      \begin{enumerate}
       \item The group $G$ is quasi-isometric to a right-angled Coxeter group. 
       \item The group $G$ is not abstractly commensurable to any group with JSJ graph a tree. In particular, $G$ is not abstractly commensurable to any group generated by finite-order elements.
      \end{enumerate}
   \end{thm}
   
   We first prove (1) in the next lemma.
   
   \begin{lemma} \label{lemma:G_qi_racg}
    The group $G$ is quasi-isometric to a right-angled Coxeter group.
   \end{lemma}
    \begin{proof}
     Let $G'$ be a geometric amalgam of free groups with JSJ graph $\Omega$ shown in Figure~\ref{figure:not_AC_tree}, where the vertex groups associated to the white vertices are infinite cyclic and the vertex groups associated to the black vertices are maximally hanging Fuchsian. The graph $\Omega$ is a tree; so, by \cite[Theorem 1.16]{danistarkthomas}, $G'$ is quasi-isometric to a right-angled Coxeter group. The group $G'$ also has degree refinement the matrix shown. Therefore, by Theorem~\ref{thm:qi_same_deg_ref}, the groups $G$ and $G'$ are quasi-isometric. 
    \end{proof}

   \begin{outline}
    We outline the proof of Theorem~\ref{thm:qi_notAC}(2) in the case of a simplifying assumption. Suppose towards a contradiction that $G$ is abstractly commensurable to a group $G'$ with JSJ graph a tree and so that $G' \cong \pi_1(X')$ where $X'$ is a $2$-dimensional hyperbolic $P$-manifold.  In this setting, by \cite[Theorem 1.2]{lafont}, there exist finite covers $p:Y \rightarrow X$ and $p':Y' \rightarrow X'$, where $Y$ and $Y'$ are $2$-dimensional hyperbolic $P$-manifolds, and there exists a homeomorphism $f:Y \rightarrow Y'$ inducing an isomorphism between finite-index subgroups of $G$ and $G'$. 
    
    The vertices $\{v_{\alpha}, v_A, v_{\beta}, v_B\}$ form a cycle in the JSJ graph $\Lambda$ for $G$. The full preimage of $A \cup B \cup \alpha \cup \beta$ in the cover $Y \rightarrow X$ yeilds (not necessarily disjoint) cycles in the JSJ graph for $\pi_1(Y)$ and hence in the JSJ graph for $\pi_1(Y') \cong \pi_1(Y)$. We show these cycles cannot project to a tree in the JSJ graph for $\pi_1(X')$. The first step is to show that $p'(f(p^{-1}(S_v))) \cap p'(f(p^{-1}(S_v'))) = \emptyset$, where $S_v$ and $S_v'$ are the closed surfaces in $X$ labeled in Figure~\ref{figure:not_AC_tree}. This claim holds since the vertex groups in $F_3^G$ are exactly $\{\pi_1(C_1), \pi_1(C_2), \pi_1(C_1'), \pi_1(C_2')\}$ and the ratio of the Euler characteristic of the subsurfaces in $S_v$ is different from the ratio of the Euler characteristic of the subsurfaces of $S_v'$. Consequently, $p'(f(p^{-1}(\alpha))) \cap p'(f(p^{-1}(\beta))) = \emptyset$. Since $\pi_1(A) \in F_1^G$ and $\pi_1(B) \in F_2^G$, Lemma~\ref{lemma:same_block} implies $p'(f(p^{-1}(A))) \cap p'(f(p^{-1}(B))) = \emptyset$. Each subsurface or curve in $X'$ in $\{p'(f(p^{-1} A ))),p'(f(p^{-1} B))),p'(f(p^{-1} \alpha))),p'(f(p^{-1}\beta )))\}$ corresponds to a vertex in a set $V'$ contained in the JSJ graph for $\pi_1(X')$. The vertices in $V'$ are adjacent to at least two other vertices in $V'$ by the above arguments, and this yeilds a cycle in the JSJ graph, a contradiction. The proof in full generality given below translates these topological ideas to the algebraic setting.
   \end{outline}

    \begin{proof}[Proof of Theorem~\ref{thm:qi_notAC}]
     The proof of (1) is given in Lemma~\ref{lemma:G_qi_racg}.
     To prove (2), suppose towards a contradiction that $G$ is abstractly commensurable to a group $G'$ with JSJ graph a tree $\Lambda'$. Suppose $H \leq G$ and $H' \leq G'$ are finite-index subgroups which are isomorphic. Since $G$ is a geometric amalgam of free groups, the groups $H$ and $H'$ are geometric amalgams of free groups. Hence, $H \cong \pi_1(Y)$ and $H' \cong \pi_1(Y')$ where $Y$ and $Y'$ are $2$-dimensional hyperbolic $P$-manifolds. By \cite[Theorem 1.2]{lafont} there exists a homeomorphism $f: Y \rightarrow Y'$ inducing an isomorphism $\Phi:H \rightarrow H'$. Suppose $p:Y \rightarrow X$ is a finite covering map. 
     
     Let $\cS$ be the full pre-image of $S_v$ in $Y$, and let $\cS'$ be the full pre-image of $S_v'$ in $Y$, where $S_v = C_1 \cup C_2$ and $S_v'=C_1' \cup C_2'$ are the closed surfaces in $X$ labeled in Figure~\ref{figure:not_AC_tree}. The spaces $\cS$ and $\cS'$ are each a disjoint collection of connected closed surfaces and $\cS \cap \cS' = \emptyset$ since $S_v \cap S_v' = \emptyset$.     
     The lifts of the branching curve $\alpha$ on $X$ partition each surface $\Sigma$ in $\cS$ as $\Sigma = Z_1 \cup Z_2$ so that $p(Z_i) = C_i$. Similarly, the lifts of the branching curve $\beta$ on $X$ partition each surface $\Sigma'$ in $\cS'$ as $\Sigma' = Z_1' \cup Z_2'$ where $p(Z_i') =C_i'$.      
     Each branching curve in $Y$ that intersects $\Sigma$ is incident to exactly one subsurface in each of $Z_1$ and $Z_2$. Hence, $Z_1$ covers $C_1$ by the same degree that $Z_2$ covers $C_2$. A similar argument holds for $\Sigma'$.
     Moreover, since $f: Y \rightarrow Y'$ is a homeomorphism, 
     \begin{eqnarray} \label{eqn:ratios}
      \frac{\chi(f(Z_1))}{\chi(f(Z_2))} = \frac{\chi(Z_1)}{\chi(Z_2)} =\frac{\chi(C_1)}{\chi(C_2)} \quad\quad &\text{ and }& \quad \quad \frac{\chi(f(Z_1'))}{\chi(f(Z_2'))} = \frac{\chi(Z_1')}{\chi(Z_2')} = \frac{\chi(C_1')}{\chi(C_2')}.
     \end{eqnarray}
     
     Let $H'_{\Sigma} \cong \pi_1(f(\Sigma)) \leq H'$ and $H'_{\Sigma'} \cong \pi_1(f(\Sigma')) \leq H'$. We set notation in this paragraph. The subgroup 
     $H'_{\Sigma}$ is generated by a union of vertex groups in $T_1^{H'}$ and $F_3^{H'}$ in the JSJ decomposition of $H'$ by Lemma~\ref{lemma:same_block}. Suppose $H'_{\Sigma} = \la H'_{w_1}, \ldots, H'_{w_r}, H'_{x_1}, \ldots, H'_{x_s} \ra$, where $w_i \in F_3(H')$ for $i \in \{1, \ldots, r\}$ and $x_j \in T_1(H')$ for $j \in \{1, \ldots, s\}$. 
     Similarly, $H'_{\Sigma'} = \la H'_{w_1'}, \ldots, H'_{w_{r'}'}, H'_{x_1'}, \ldots, H'_{x_{s'}'} \ra$, where $w_i' \in F_3(H')$ for $i \in \{1, \ldots, r'\}$ and $x_j' \in T_1(H')$ for $j \in \{1, \ldots, s'\}$. 
     Assume that the fundamental group of every branching curve in $Y'$ that intersects $f(\Sigma)$ is included in the set $\{H'_{x_j}\}_{j=1}^s$, and assume that the fundamental group of every branching curve in $Y'$ that intersects $f(\Sigma')$ is included in the set $\{H'_{x_j'}\}_{j=1}^{s'}$. That is, the generating sets above are not minimal since these branching curves are boundary curves of surfaces whose fundamental groups are contained in the sets $\{H_{w_i}'\}_{i=1}^r$ and $\{H_{w_i'}'\}_{i=1}^{r'}$.      
     By Proposition~\ref{prop:sub_struc} and Lemma~\ref{lemma:same_block}, there are vertices $v_i, v_i' \in F_3(G')$ and $y_j,y_j' \in T_1(G')$ and elements $g_{w_i},g_{w_i'},g_{x_j},g_{x_j'} \in G'$ so that 
     \begin{eqnarray*}
      H'_{w_i} = H' \cap g_{w_i}G_{v_i}'g_{w_i}^{-1} &\quad \quad \quad&
       H'_{x_j} = H' \cap g_{x_j}G_{y_j}'g_{x_j}^{-1}\\
       H'_{w_i'} = H' \cap g_{w_i'}G_{v_i'}'g_{w_i'}^{-1} &\quad \quad \quad&
       H'_{x_j'} = H' \cap g_{x_j'}G_{y_j'}'g_{x_j'}^{-1}.
     \end{eqnarray*}     
     The surfaces $f(\Sigma)$ and $f(\Sigma')$ are disjoint in the space $Y'$, so \[\{w_1, \ldots, w_r, x_1, \ldots, x_s\} \cap \{w_1', \ldots, w'_{r'}, x_1', \ldots, x'_{s'} \} = \emptyset.\]
     
     \noindent {\it Claim:} $\{v_1, \ldots, v_r, y_1, \ldots, y_s\} \cap \{v_1', \ldots, v'_{r'}, y_1', \ldots, y'_{s'}\} = \emptyset$.
     \begin{proof}[Proof of Claim.]    
     Since the fundamental group of every branching curve in $Y'$ that intersects $f(\Sigma)$ or $f(\Sigma')$ is included in the set $\{H_{x_j}\}_{j=1}^s \cup \{H_{x_{j'}}\}_{j=1}^{s'}$, if $v_i = v_{\ell}'$ for some $i, \ell$, then $y_j = y_k'$ for some $j,k$. 
     Suppose towards a contradiction $y_j = y_k'$ for some $j \in \{1, \ldots, s\}$ and $k \in \{1, \ldots, s'\}$. 
     
     We first show recursively this assumption implies $\{v_1, \ldots, v_r, y_1, \ldots, y_s\} = \{v_1', \ldots, v'_{r'}, y_1', \ldots, y'_{s'}\}$.   
     Each vertex $x_j \in \{x_i\}_{i=1}^s$ is adjacent to exactly two vertices $w,w' \in \{w_i\}_{i=1}^r$.
     By the structure of the degree refinement for $H'$, each vertex $x_j \in \{x_i\}_{i=1}^s$ is not adjacent to any other vertices in $F_3(H')$. An analogous statement holds for the vertices $\{x_j'\}_{j=1}^{s'}$ and $\{w_i'\}_{i=1}^{r'}$. Since the degree refinement for $G'$ is the same as the degree refinement for $H'$, an analogous statement also holds for the pairs $(\{v_i\}_{i=1}^r, \{y_j\}_{j=1}^s)$ and $(\{v_i'\}_{i=1}^{r'}, \{y_j'\}_{j=1}^{s'})$. 
     Therefore, if $y_j = y_k'$, then the two vertices $v, v' \in F_3(G')$ incident to $y_j = y_k'$ are in the sets $\{v_i\}_{i=1}^r$ and $\{v_i'\}_{i=1}^{r'}$. 
     Since $f(\Sigma)$ and $f(\Sigma')$ are closed surfaces, for each vertex $w_i$ and $w_i'$ with $i \in \{1, \ldots, r\}$ and $i' \in \{1, \ldots, r'\}$, each vertex in $T_1(H')$ incident to $w_i$ is in the set $\{x_j\}_{j=1}^s$ and each vertex in $T_1(H')$ incident to $w_i'$ is in the set $\{x_j'\}_{j=1}^{s'}$. Thus, an analogous statement holds for the pairs $(\{v_i\}, \{y_i\})$ and $(\{v_i'\}, \{y_i'\})$. So, each vertex in $T_1(G')$ incident to either $v$ or $v'$ is in the sets $\{y_i\}_{i=1}^r$ and $\{y_i'\}_{i=1}^{r'}$. The above argument can then be applied to these vertices incident to $v$ and $v'$. Each pair of vertex sets $(\{w_i\}, \{x_j\})$, $(\{w_i'\}, \{x_j'\})$, $(\{v_i\}, \{y_j\})$, and $(\{v_i'\}, \{y_j'\})$ spans a connected subgraph in either $\Gamma'$ or $\Lambda'$. Therefore, these arguments can be repeated to conclude $\{v_1, \ldots, v_r, y_1, \ldots, y_s\} = \{v_1', \ldots, v'_{r'}, y_1', \ldots, y'_{s'}\}$. 
     
     The partitions of the surfaces $f(\Sigma) = f(Z_1) \cup f(Z_2)$ and $f(\Sigma') = f(Z_1') \cup f(Z_2')$ yield partitions of the vertices  $\{w_i\}_{i=1}^r = \cZ_1 \sqcup \cZ_2$ and $\{w_i'\}_{i=1}^{r'} = \cZ_1' \sqcup \cZ_2'$. Hence, there are partitions $\{v_i\}_{i=1}^r = \cV_1 \sqcup \cV_2$ and $\{v_i'\}_{i=1}^{r'} = \cV_1' \sqcup \cV_2'$, where $v_i \in \cV_j$ if $H_{w_i}$ is the fundamental group of a subsurface in $f(Z_j)$, and likewise for $v_i'$. Each vertex $x_j \in \{x_k\}_{k=1}^s$ is adjacent to exactly one vertex in $\cZ_1$ and exactly one vertex in $\cZ_2$, and an analogous statement holds for $x_j'\in \{x_k'\}_{k=1}^{s'}$. Hence, each vertex $y_j$ for $j \in \{1, \ldots, s\}$ is adjacent to exactly one vertex in each of $\cV_1$ and $\cV_2$, and the same holds for $y_j'$ for $j \in \{1, \ldots, s'\}$. Therefore, by the conclusion of the previous paragraph, either $\cV_1 = \cV_1'$ and $\cV_2 = \cV_2'$ or $\cV_1 = \cV_2'$ and $\cV_2 = \cV_1'$. Assume $\cV_1 = \cV_1'$ and $\cV_2 = \cV_2'$; the other case is similar. 
    
     Since each $y_j$ with $j \in \{1, \ldots, s\}$ is incident to exactly one vertex in each of $\cV_1$ and $\cV_2$, by the description of finite-index subgroups in \cite[Section 3]{scottwall} and Lemma~\ref{lemma:chi_mult}, there exists $d \in \N$ so that \[d\cdot \sum_{v \in \cV_1} \chi(G_v') = \chi(f(Z_1)) \quad \text{ and } \quad d\cdot\sum_{v \in \cV_2} \chi(G_v') = \chi(f(Z_2)). \] Similarly, there exists $d' \in \N$ so that \[d'\cdot\sum_{v \in \cV_1'} \chi(G_v') = \chi(f(Z_1')) \quad \text{ and } \quad d'\cdot\sum_{v \in \cV_2'} \chi(G_v') = \chi(f(Z_2')). \]      
      Therefore, since $\cV_1 = \cV_1'$ and $\cV_2 = \cV_2'$,
      \[\frac{\chi(f(Z_1))}{\chi(f(Z_2))} = \frac{\sum_{v \in \cV_1} \chi(G_v')}{\sum_{v \in \cV_2} \chi(G_v')} = \frac{\sum_{v \in \cV_1'} \chi(G_v')}{\sum_{v \in \cV_2'} \chi(G_v')} = \frac{\chi(f(Z_1'))}{\chi(f(Z_2'))}.\]
      So, by Equation~\ref{eqn:ratios}, $\frac{\chi(C_1)}{\chi(C_2)} = \frac{\chi(C_1')}{\chi(C_2')}$, a contradiction. Therefore, \[\{v_1, \ldots, v_r, y_1, \ldots, y_s\} \cap \{v_1', \ldots, v'_{r'}, y_1', \ldots, y'_{s'}\} = \emptyset.\]      
     \end{proof}
     
     To conclude the proof of the theorem, let $\cH_{\alpha}' \subset V(\Gamma')$ and $\cH_{\beta}' \subset V(\Gamma')$ be the sets of vertices in $\Gamma'$ whose vertex groups are the fundamental group of a component of $f(p^{-1}(\alpha))$ and $f(p^{-1}(\beta))$, respectively, where $\alpha$ and $\beta$ are the singular curves in $X$ labeled in Figure~\ref{figure:not_AC_tree}. Let $\cG_{\alpha}' \subset \Lambda'$ and $\cG_{\beta}' \subset \Lambda'$ be the set of vertices of $\Lambda'$ whose vertex groups contain, as a finite-index subgroup, $H'_w$ for $w \in \cH_{\alpha}'$ and $w \in \cH_{\beta}'$, respectively. The above arguments imply $\cG_{\alpha}' \cap \cG_{\beta}' = \emptyset$. Indeed, every lift of $\alpha$ in $Y$ is contained in some surface $\Sigma$ in $\cS$, and every lift of $\beta$ in $Y$ is contained in some surface $\Sigma'$ in $\cS'$. Let $\cH_A' \subset V(\Gamma')$ and $\cH_B' \subset V(\Gamma')$ be the set of vertices in $\Gamma'$ whose vertex groups are the fundamental group of a component of $f(p^{-1}(A))$ and $f(p^{-1}(B))$, respectively, where $A$ and $B$ are the subsurfaces in $X$ shown in Figure~\ref{figure:not_AC_tree}. Let $\cG_A' \subset \Lambda'$ and $\cG_B' \subset \Lambda'$ be the set of vertices in $\Lambda'$ whose vertex groups contain, as a finite-index subgroup $H_w'$ for $w \in \cH_A'$ and $w \in \cH_B'$, respectively. Since $\pi_1(A) \in F_1^G$ and $\pi_1(B) \in F_2^G$, by Lemma~\ref{lemma:same_block}, $\cG_A' \cap \cG_B' = \emptyset$. Every vertex in $\cH_A'$ is adjacent to a vertex in $\cH_{\alpha}'$ and a vertex in $\cH_{\beta}'$ and vice-versa. Similarly, every vertex in $\cH_B'$ is adjacent to a vertex in $\cH_{\alpha}'$ and a vertex in $\cH_{\beta}'$ and vice-versa. Therefore, by Lemma~\ref{lemma:adj_vert}, every vertex in $\cG_A'$ and $\cG_B'$ is adjacent to a vertex in $\cG_{\alpha}'$ and a vertex in $\cG_{\beta}'$ and vice-versa. So, there is a cycle in $\Lambda'$, a contradiction. Therefore, $G$ is not abstractly commensurable to any group with JSJ graph a tree.  
     \end{proof}

  \section{Necessary conditions for commensurability} \label{sec:comm_cond}

  \subsection{Block Euler characteristic vector}

  \begin{defn} 
    Let $G \in \cC$, and let $\Lambda$ be the JSJ graph of $G$ so that $V(\gL) = V_1 \sqcup V_2$ as in Definition~\ref{def:deg_ref_G}. Suppose in the degree partition of $\Lambda$, the vertices in $V_1$ are contained in blocks $T_1, \ldots, T_n$ and the vertices in $V_2$ are contained in blocks $F_1, \ldots, F_m$. Let $\chi_i = \displaystyle\sum_{G_v \in F_i^G} \chi(G_v)$. Suppose the blocks $\{F_i\}_{i=1}^m$ are indexed such that $\chi_i \geq \chi_j$ for $i \geq j$.
    The {\it block Euler characteristic vector of $G$} is $(\chi_1, \ldots, \chi_m)$.
  \end{defn}
  
  \begin{defn}
   Vectors $v, v' \in \R^n$ are {\it commensurable} if there exist non-zero integers $K, K' \in \Z$ so that $Kv = K'v'$. 
  \end{defn}  
  
  \begin{prop} \label{prop_block_comm}
   If $G, G' \in \cC$ are abstractly commensurable, then the block Euler characteristic vector of $G$ is commensurable to the block Euler characteristic vector of $G'$. 
  \end{prop}
  \begin{proof}
    Let $(\chi_1, \ldots, \chi_m)$ and $(\chi_1', \ldots, \chi_m')$ be the block Euler characteristic vectors of $G$ and $G'$, respectively, where $\chi_i = \displaystyle \sum_{G_v \in F_i^G} \chi(G_v)$ and $\chi_i' = \displaystyle\sum_{G_{v'}'\in F_i^{G'}} \chi(G_{v'}')$.
    Suppose $H \leq G$ and $H' \leq G'$ are finite-index subgroups with $H \cong H'$. Suppose $[G:H]=d$ and $[G':H'] = d'$. By Lemma~\ref{lemma:sum_chi}, for every $v \in V(\gL)$ and $v' \in V(\gL')$,
    \[d\cdot \chi(G_v) = \sum_{H_w \in \cH_v} \chi(H_w) \quad \text{ and } \quad d'\cdot \chi(G'_{v'}) = \sum_{H'_{w'} \in \cH'_{v'}} \chi(H'_{w'}). \]
    Furthermore, by Lemma~\ref{lemma:same_block},
    \[\bigcup_{v \in F_i^G} \cH_v = F_i^H \quad \text{ and } \quad \bigcup_{v' \in F_i^{G'}} \cH'_{v'} = F_i^{H'}.\]
    Therefore,
    \[ d\cdot \sum_{G_v \in F_i^G}  \chi(G_v)
    \,\, = \,\, \sum_{H_w \in F_i^H} \chi(H_w)
    \,\, = \,\, \sum_{H'_{w'} \in F_i^{H'}} \chi(H'_{w'})
    \,\, = \,\, d'\cdot \sum_{G_v' \in F_i^{G'}} \chi(G_v'). \]
    So, $d (\chi_1, \ldots, \chi_m) = d' (\chi_1', \ldots, \chi_m')$. 
  \end{proof}

  \subsection{Matching Euler characteristic vector}

  Suppose $G \in \cC$ has JSJ decomposition with underlying graph $\Lambda$ with $V(\gL)= V_1 \sqcup V_2$ as in Definition~\ref{def:deg_ref_G}. If all vertices in $V_1$ have the same valance, then the block Euler characteristic vector of $G$ has one entry. Hence, the block Euler characteristic vector does not distinguish commensurability classes in this setting. So, we define a finer invariant in this subsection to deal with this case. Results in this section are proved only within the subclass of geometric amalgams of free groups. 

  If $\Gamma$ is a graph, a {\it matching} in $\Gamma$ is a collection of disjoint edges whose vertex set is exactly the vertex set of $\Gamma$. This notion extends to subsurfaces in a $2$-dimensional hyperbolic $P$-manifold, and a necessary criterion for commensurability can be stated in these terms. 
  
  \begin{defn}
   Let $X$ be a $2$-dimensional hyperbolic $P$-manifold. A {\it matching} in $X$ is a (not necessarily connected) subsurface in $X$ (see Definition~\ref{def:geom_amal}) whose boundary is exactly the set of branching curves in $X$. In particular, a matching does not contain any branching curves in its interior, and each branching curve is incident to exactly one connected subsurface with boundary in the matching. A {\it maximal matching} in $X$ is a matching which has the greatest Euler characteristic of any matching in~$X$. 
  \end{defn}

  \begin{lemma}[Existence of a matching]
   Let $X$ be a possibly disconnected hyperbolic $P$-manifold over a finite forest, and suppose each branching curve in $X$ is incident to exactly $n$ subsurfaces. Then $X$ admits a matching. 
  \end{lemma}
  \begin{proof}
    Let $X$ be a possibly disconnected hyperbolic $P$-manifold over a finite forest $T$, and suppose each branching curve in $X$ is incident to exactly $n$ subsurfaces.
    Suppose $X$ has branching curves $c_1, \ldots, c_m$. To build a matching $\cM$ of $X$, choose $\Sigma_1 \subset X$, a subsurface of $X$, and let $\Sigma_1 \in \cM$. Without loss of generality, $\Sigma_1$ has boundary $c_1, \ldots, c_k$ for some $k \leq m$. 
   Let $X_1 \subset X$ be the set of subsurfaces in $X$ which have at least one boundary component in $\{c_1, \ldots, c_k\}$.
   Without loss of generality, the set of boundary curves of subsurfaces in $X_1$ is $\{c_1, \ldots, c_{k+\ell}\}$.
   Let $X_1' = \overline{X \setminus X_1}$, a hyperbolic $P$-manifold over a finite forest.    
   Since $T$ is a finite forest, each curve in $\{c_{k+1}, \ldots, c_{k+\ell}\}$ has degree $n-1>1$ in $X_1'$. So, it is possible to add to $\cM$ one subsurface in $X_1'$ incident to each curve in $\{c_{k+1}, \ldots, c_{k+\ell}\}$. Repeat this procedure with each new surface chosen for $\cM$ in the place of $\Sigma_1$ to produce a matching $\cM$ in finitely many steps.     
  \end{proof}

  \begin{example} \label{example:no_matching}
   There are $2$-dimensional hyperbolic $P$-manifolds for which each branching curve has the same degree and for which there does not exist a matching. Indeed, let $X$ be a $2$-dimensional hyperbolic $P$-manifold whose fundamental group has JSJ graph $\Omega$ shown in Figure~\ref{figure:no_matching}. The white vertices in $\Omega$ correspond to $2$-ended vertex groups; the black vertices correspond to maximally hanging Fuchsian vertex groups.
  \end{example}
  
       \begin{figure}
      \begin{overpic}[scale=.3, tics=5]{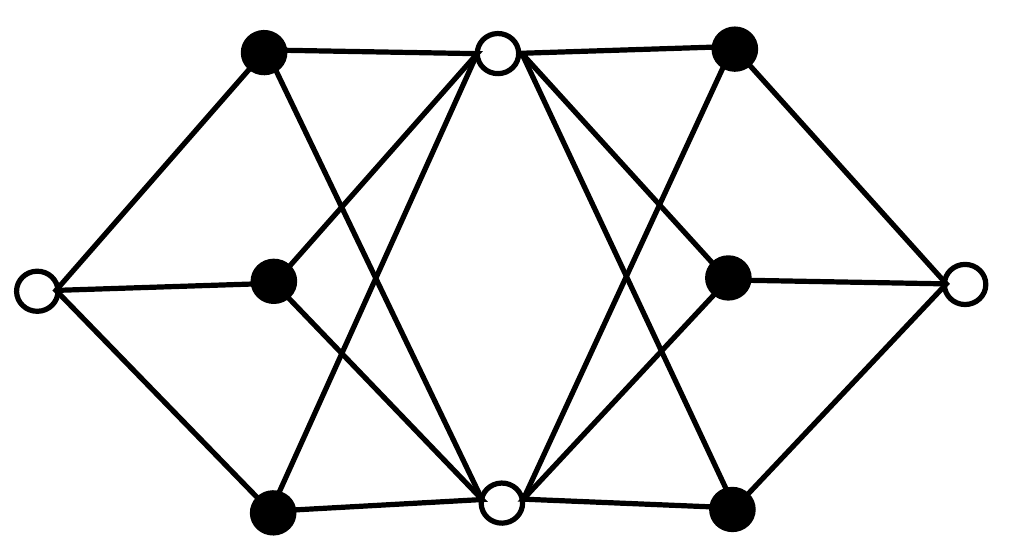}
      \put(0,40){$\Omega$}
        \end{overpic}
	\caption{{\small The JSJ graph $\Omega$ of a $2$-dimensional hyperbolic $P$-manifold whose fundamental group does not admit a matching. }}
      \label{figure:no_matching}
     \end{figure}

  \begin{defn}
   Let $X$ be a $2$-dimensional hyperbolic $P$-manifold with underlying graph a tree, and suppose each branching curve in $X$ is incident to exactly $n$ subsurfaces. Let $\cM_1$ be a maximal matching in $X$, and let $\cM_i$ be a maximal matching in $X \setminus \Int(\bigcup_{k=1}^{i-1} \cM_i)$ for $i=2, \ldots , n$. (Note that in a slight abuse of notation, we still refer to the image of the branching curves of $X$ in $X \setminus \Int(\bigcup_{k=1}^{n-2} \cM_i)$ as ``branching'' even though at this step the space is a manifold; similarly for $X \setminus \Int(\bigcup_{k=1}^{n-1} \cM_i) $.)    
   The {\it matching Euler characteristic vector} of $X$ is
   \[(\chi(\cM_1), \chi(\cM_2), \ldots, \chi(\cM_n)). \]
  \end{defn}
  
      The proof of the following proposition generalizes \cite[Proposition 3.3.2]{stark} and \cite[Proposition 6.4]{danistarkthomas}
  
  \begin{prop} \label{prop_matching_comm}
     Suppose $X$ and $X'$ are $2$-dimensional hyperbolic $P$-manifolds, the underlying graphs of $X$ and $X'$ are trees, and each branching curve in $X$ and $X'$ is incident to exactly $n$ subsurfaces with boundary. Let $v$ and $v'$ be the matching Euler characteristic vectors of $X$ and $X'$, respectively. If $\pi_1(X)$ and $\pi_1(X')$ are abstractly commensurable, then $v$ and $v'$ are commensurable vectors. 
  \end{prop}
  
  \begin{proof}
   Suppose that $X$ and $X'$ are $2$-dimensional hyperbolic $P$-manifolds, the underlying graphs of $X$ and $X'$ are trees, and each branching curve in $X$ and $X'$ is incident to $n$ branching curves. Let $v = (\chi(\cM_1), \ldots, \chi(\cM_n))$ and $v' = (\chi(\cM_1'), \ldots, \chi(\cM_n'))$ be the matching Euler characteristic vectors of $X$ and $X'$, respectively. Suppose that $\pi_1(X)$ and $\pi_1(X')$ are abstractly commensurable. We seek to prove that $v$ and $v'$ are commensurable. 
   
   Since $\pi_1(X)$ and $\pi_1(X')$ are abstractly commensurable, there are finite covering spaces $p:Y \rightarrow X$ and $p':Y' \rightarrow X'$ and a homeomorphism $f:Y \rightarrow Y'$ by \cite[Theorem 1.2]{lafont}. Suppose that $p$ is a degree $D$ cover and $p'$ is a degree $D'$ cover. We will show $Dv = D'v'$. 
   
   Suppose that  
   \begin{center}
    $\chi(\cM_1) = \ldots = \chi(\cM_s) > \chi(\cM_{s+1})\geq \ldots \geq \chi(\cM_{n}) $,
    
    $\chi(\cM_1') = \ldots = \chi(\cM_t') > \chi(\cM_{t+1}')\geq \ldots \geq \chi(\cM_{n}') $.
   \end{center}

   Without loss of generality, we may assume that $D\chi(\cM_1) \geq D'\chi(\cM_1')$ and if $D(\chi(\cM_1) = D'\chi(\cM_1')$ then $s \geq t$. 
   
   Consider the matching $f(p^{-1}(\cM_1)) = \{S_1', \ldots, S_r'\} \subset Y'$, a disjoint collection of connected subsurfaces of $Y'$ whose boundary is exactly the set of branching curves of $Y'$. Let $\{c_1, \ldots, c_m\}$ be the set of branching curves of $X'$. Each surface $S_i' \in f(p^{-1}(\cM_1))$ covers a subsurface $S_{I_i}$ of $X'$, where $S_{I_i}$ has boundary $\{c_j \,|\, j \in I_i\}$ for some $I_i \subset \{1, \ldots, m\}$. Suppose $S_i'$ covers $S_{I_i}$ by degree $d_{I_i}$ with $1 \leq d_{I_i} \leq D'$. 
   
   \noindent {\it Claim:} $\displaystyle \sum_{i=1}^r d_{I_i}\cdot \chi(S_{I_i}) \leq D'\chi(\cM_1')$.
   \begin{proof}[Proof of Claim.]
    The inequality holds by the definition of maximal matching if $D' = 1$. In general, the surfaces in $\{S_{I_i}\}_{i=1}^r$ need not form a matching of $X'$. However, we show that if these surfaces are counted with the right multiplicity, then they can be partitioned into $D'$ matchings of $X'$. 
    
    If $D' >1$, we will partition the surfaces in the set
    \[ \cS = \{\underbrace{S_{I_1}, \ldots, S_{I_1}}_{d_{I_1}}, \underbrace{S_{I_2}, \ldots, S_{I_2}}_{d_{I_2}}, \ldots, \underbrace{S_{I_r}, \ldots, S_{I_r}}_{d_{I_r}} \} \] into $D'$ matchings of $X'$ called $\cN_1, \ldots, \cN_{D'}$. The sum of the Euler characteristics of the surfaces in the $D'$ matchings constructed is equal to the left-hand side of the inequality in the claim. The conclusion of the claim will follow by the definition of maximal matching. For $i \in \{1,\ldots, r\}$, we call the value $d_{I_i}$ the {\it weight} of the surface $S_{I_i}$. The weight of $S_{I_i}$ records the contribution of $\chi(S_{I_i})$ to the Euler characteristic of $f(p^{-1}(\cM_1))$. If $c_j$ is a boundary curve of the surface $S_{I_i}$, we say $d_{I_i}$ is the weight at $c_j$ coming from $S_{I_i}$. The total weight at the curve $c_j$ is $\sum_{j \in I_i} d_{I_i} = D'$ since $p':Y' \rightarrow X'$ is a degree-$D'$ cover. 
    
    Construct the set of $D'$ matchings recursively. To construct the first matching $\cN_1$ built out of surfaces in $\cS$, choose a surface $S_{I_i} \in \cS$. The surface $S_{I_i}$ is incident to curves $\{c_j \,|\, j \in I_i\}$. View the underlying graph of $X$ as a bipartite graph with white vertices corresponding to branching curves and black vertices corresponding to surfaces with boundary. Let $c$ be a curve corresponding to a white vertex at distance two (in the graph) from $c_j$ for some $j \in I_i$ and at distance three (in the graph) from the black vertex corresponding to $S_{I_i}$, if such a curve exists. There is a surface in $\cS$ incident to $c$ and not to $c_j$. That is, both curves $c$ and $c_j$ have total weight $D'$ coming from a collection of surfaces incident to these curves. Some of the weight at $c_j$ comes from $S_{I_i}$, so the (unique) surface $S$ incident to both $c_j$ and $c$ has weight less than $D'$. Thus, some of the weight at $c$ comes from a surface $S_{I_k}$ different from $S$. Let $S_{I_i} \cup S_{I_k} \in \cN_1$. Since the underlying graph of $X$ is a tree, this selection may be continued (finitely many times) to build the matching $\cN_1$ of surfaces contained in $\cS$. 
   
    Build the remaining $D'-1$ matchings $\cN_2, \ldots, \cN_{D'}$ similarly. First form the set $\cS_1$ by removing from $\cS$ one copy of each surface contained in $\cN_1$. Then, the matching $\cN_2$ may be chosen analogously to $\cN_1$, where each branching curve in $X'$ now has weight $D'-1$ associated to it. This process may be continued to build the $D'$ matchings $\cN_1, \ldots, \cN_{D'}$. Furthermore, by construction and since $\cM_1'$ is a maximal matching of $X'$, 
    \[\sum_{i=1}^r d_{I_i}\cdot\chi(S_{I_i}) =     \chi(\cN_1) + \ldots + \chi(\cN_{D'}) \leq D'\chi(\cM_1')  \] concluding the proof of the claim.    
   \end{proof}
   
   Thus, 
   \[D\chi(\cM_1) \,\,=\,\, \chi(f(p^{-1}(\cM_1)))
    \,\,=\,\, \sum_{i=1}^r d_{I_i}\cdot\chi(S_{I_i}) 
    \,\,\leq \,\, D'\chi(\cM_1'), \] where the last inequality is given by the claim above. Since $D\chi(\cM_1) \geq D'\chi(\cM_1')$ by assumption, $D\chi(\cM_1) = D'\chi(\cM_1')$. Each branching curve in $Y'$ is incident to exactly $s$ connected surfaces in $f(p^{-1}(\cM_1)) \cup \ldots \cup f(p^{-1}(\cM_s))$. Thus, $p'(f(p^{-1}(\cM_1)) \cup \ldots \cup f(p^{-1}(\cM_s)) )$ must have in its image at least $s$ surfaces in the matchings $\cM_1', \ldots, \cM_n'$; so, $t \geq s$. Therefore, $D\chi(\cM_i) = D'\chi(\cM_i')$ for $1 \leq i \leq s=t$ and $\bigcup_{i=1}^s f(p^{-1}(\cM_i)) = \bigcup_{i=1}^s p^{-1}(\cM_i)$. So, the above argument can be repeated (at most finitely many times) with the remaining matchings in $X$ and $X'$ of strictly smaller Euler characteristic, proving the proposition. 
  \end{proof}


    \section{Quasi-isometry versus abstract commensurability} \label{sec:qi_vs_ac}

   \begin{thm}
    There are infinitely many abstract commensurability classes within every quasi-isometry class in $\cC$ that contains a group with JSJ graph a tree. 
   \end{thm}
    \begin{proof}
      Let $\cQ$ be a quasi-isometry class in $\cC$ that contains a group with JSJ graph a tree $T$. Let $G \in \cQ$ be a geometric amalgam of free groups with JSJ graph $T$ so that $G \cong \pi_1(X)$ for some $2$-dimensional hyperbolic $P$-manifold $X$. Choose a subsurface $\Sigma \subset X$. Exchange $\Sigma$ with a surface $\Sigma_g$ with the same number of boundary components as $\Sigma$ and with genus $g \geq 1$. Let $X_g$ denote the resulting $2$-dimensional hyperbolic $P$-manifold, and let $G_g \cong \pi_1(X_g)$. The groups $G_g$ and $G$ are quasi-isometric for every $g \geq 1$, but $G_g$ and $G_h$ are abstractly commensurable if and only if $g=h$. Indeed, if not all branching curves of $X_g$ have the same degree, let $v_g$ be the block Euler characteristic vector of $G_g$, and      
      if all branching curves of $X_g$ have the same degree, let $v_g$ be the  matching Euler characteristic vector of $G_g$. Changing the Euler characteristic of one subsurface of $X_g$ changes the commensurability type of the vector $v_g$, hence the claim follows from Proposition~\ref{prop_block_comm} and Proposition~\ref{prop_matching_comm}. 
    \end{proof}

    \begin{remark}
     The proof above applies to any quasi-isometry class in $\cC$ for which the degree refinement contains more than one row coming from maximal hanging Fuchsian vertex groups, since in this case the block Euler characteristic vector has more than one entry. 
    \end{remark}

\bibliographystyle{alpha}
\bibliography{refs}

\end{document}